\documentclass[11pt,leqno,twoside]{amsart}
\usepackage{amssymb,amsmath,amsthm,soul,color}
\usepackage{t1enc}
\usepackage[cp1250]{inputenc}
\usepackage{a4,indentfirst,latexsym,graphics,graphicx}
\usepackage{graphics}
\usepackage{mathrsfs}
\usepackage{cite,enumitem,graphicx}
\usepackage[colorlinks=true,urlcolor=blue,
citecolor=red,linkcolor=blue,linktocpage,pdfpagelabels,
bookmarksnumbered,bookmarksopen]{hyperref}
\usepackage[english]{babel}
\usepackage[left=2.1cm,right=2.1cm,top=2.72cm,bottom=2.72cm]{geometry}
\usepackage[metapost]{mfpic}
\usepackage[hyperpageref]{backref}
\usepackage[colorinlistoftodos]{todonotes}
\usepackage[normalem]{ulem}

\makeatletter
\providecommand\@dotsep{5}
\def\listtodoname{List of Todos}
\def\listoftodos{\@starttoc{tdo}\listtodoname}
\makeatother

\numberwithin{equation}{section}
\newtheorem{Th}{Theorem}[section]
\newtheorem{Prop}[Th]{Proposition}
\newtheorem{Lem}[Th]{Lemma}
\newtheorem{lemma}[Th]{Lemma}
\newtheorem{Cor}[Th]{Corollary}

\newtheorem{Rem}[Th]{Remark}

\title[Hartree-Fock type systems]{
Hartree-Fock  type systems:\\
existence of ground states and asymptotic behavior
}
\author[P. d'Avenia]{Pietro d'Avenia}
\author[L. A. Maia]{Liliane de Almeida Maia}
\author[G. Siciliano]{Gaetano Siciliano}
\address[P. d'Avenia]{\newline\indent
	Dipartimento di Meccanica, Matematica e Management
	\newline\indent 
	Politecnico di Bari
	\newline\indent
	Via Orabona 4,  70125  Bari, Italy}
\email{\href{mailto:pietro.davenia@poliba.it}{pietro.davenia@poliba.it}}

\address[L.A. Maia]{\newline\indent
	Departamento de Matem\'atica 
	\newline\indent 
         Universidade de Bras\'ilia
	\newline\indent
	  70910-900 Bras\'ilia, Brazil}
\email{\href{mailto:lilimaia@unb.br}{lilimaia@unb.br}}

\address[G. Siciliano]{\newline\indent
	Departamento de Matem\'atica - Instituto de Matem\'atica e Estat\'istica
	\newline\indent 
	Universidade de S\~ao Paulo
	\newline\indent
	Rua do Mat\~ao 1010,  05508-090  S\~ao Paulo, Brazil}
\email{\href{mailto:sicilian@ime.usp.br}{sicilian@ime.usp.br}}


\begin{document}
	\begin{abstract} 
		In this paper we consider an Hartree-Fock type system made by two Schr\"odinger equations in presence of 
		a Coulomb interacting term and a {\em cooperative} pure power and subcritical nonlinearity,
		driven by a suitable parameter $\beta\geq 0$. We show the existence of semitrivial and vectorial 
		ground states solutions depending on the parameters involved. The asymptotic behavior with respect to the parameter $\beta$ of these solutions is also studied.
	\end{abstract}

\maketitle
\begin{center}
\begin{minipage}{12cm}
\tableofcontents
\end{minipage}
\end{center}

\section{Introduction}

In the study of a molecular system made of $M$ nuclei interacting via the Coulomb potential with $N$ electrons, the starting point is the $(M+N)$-body Schr\"odinger equation
\[
i\hbar \partial_t \Psi
=
-\frac{\hbar^2}{2} \sum_{j=1}^{M+N} \frac{1}{m_j}\Delta_{x_j} \Psi
+ \frac{e^2}{8\pi\varepsilon_0}\sum_{\substack{j, k=1 \\ j\neq k}}^{M+N} \frac{Z_j Z_k}{|x_j - x_k|} \Psi,
\qquad
\Psi: \mathbb{R}\times\mathbb{R}^{3(M+N)} \to \mathbb{C}
\]
where the constants $e Z_j$'s are the charges and in particular the charge numbers $Z_j$'s are positive for the nuclei and $-1$ for the electrons.

%\[
%i\hbar \partial_t \Psi
%=
%-\frac{\hbar^2}{2m} \sum_{j=1}^{M} \Delta_{x_j} \Psi
%+ \sum_{1\leq j < k \leq M} \frac{C}{|x_j - x_k|} \Psi,
%\qquad
%\Psi: \mathbb{R}\times\mathbb{R}^{3M} \to \mathbb{C}^N.
%\]
Its complexity led to consider various approximations to describe the stationary states with simpler models.\\
A possible approximation, used in particular in models of Quantum Chemistry, is the Born-Oppenheimer approximation. Here the nuclei are considered as classical point particles and a fundamental assumption is that they are much heavier than electrons (see e.g. \cite{CDS} for a mathematical treatment).\\
Starting from the Born-Oppenheimer model, a further possible approximation is the Hartree-Fock method, which is generally considered fundamental to much of electronic structure theory and represents the basis of molecular orbital theory.
%which posits that each electron’s motion can be described by a single-particle function (orbital) which does not depend explicitly on the instantaneous motions of the other electrons.
It is variational and
%the approach followed is that of a single-particle picture: 
the electrons are considered as occupying single-particle orbitals making up the wavefunction. Each electron feels the presence of the other electrons indirectly through an effective potential. Thus, each orbital is affected by the presence of electrons in other orbitals.\\
This was introduced by Hartree in \cite{H} through the use of some particular test functions, without taking into account the Pauli principle. Subsequently, Fock in \cite{F} and Slater in \cite{Sl}, to take into account the Pauli principle, chose a different class of test functions, the Slater determinants, obtaining a system of $N$ coupled nonlinear Schr\"odinger equations
\[
-\frac{\hbar^2}{2m} \Delta \psi_k
+V_{\rm ext} \psi_k
+ \Big(\int_{\mathbb R^{3}} |x-y|^{-1} \sum_{j=1}^{N} |\psi_j(y)|^2dy\Big)\psi_k
+ (V_{\rm ex} \psi)_k
= E_k\psi_k,
\quad
k=1,\ldots,N,
\]
where $\psi_k: \mathbb{R}^{3} \to \mathbb{C}$, $V_{\rm ext}$ is a given external potential,
\begin{equation*}
	\label{excterm}
	(V_{\rm ex} \psi)_k:=-\sum_{j=1}^{N}\psi_j\int_{\mathbb R^{3}}\frac{\psi_k(y)\overline{\psi}_j(y)}{|x-y|}dy
\end{equation*}
is the $k$'th component of the {\em crucial exchange term}, and $E_k$ is the $k$'th eigenvalue.\\
A further relevant approximation for the exchange potential $V_{\rm ex} \psi$ is due to Slater in \cite{Sl2} (see also Dirac in \cite{D} in a different context), namely
\begin{equation}
	\label{VexSlater}
	(V_{\rm ex} \psi)_k \approx -C \Big(\sum_{j=1}^{N} |\psi_j|^2 \Big)^{1/3} \psi_k.
\end{equation}
Moreover, slightly different local approximations have been done in \cite{G,KS}. For further models we refer to \cite{PY} and references therein.\\
We emphasize that in these last approximations there is a strong dependence on the electron density function $\sum_{j=1}^{N} |\psi_j|^2$.\\
For more details about the Hartree-Fock method we refer the reader to \cite{BLS,DfLB,FK,Mauser,McW,PY,SO} and references therein, and, for a mathematical approach to \cite{LS,Lions87}.

In this paper we take $N=2$
%Helium, for instance
and we assume
\begin{equation}
	\label{Vexnostro}
	(V_{\rm ex} \psi)
	= -C
	\left(\begin{matrix}
		\vspace{5pt}|\psi_1|^{q-2}\psi_1 & \beta |\psi_1|^{q-2}\psi_1\\
		\beta |\psi_2|^{q-2}\psi_2 & |\psi_2|^{q-2}\psi_2
	\end{matrix}\right)
	\left(\begin{matrix}
		\vspace{5pt}|\psi_1|^{q}\\
		|\psi_2|^{q}
	\end{matrix}\right)
	=
	-C
	\left(\begin{matrix}
		\vspace{5pt}|\psi_1|^{2q-2}\psi_1 +\beta |\psi_1|^{q-2}|\psi_2|^{q}\psi_1\\
		|\psi_2|^{2q-2}\psi_2 + \beta |\psi_1|^{q} |\psi_2|^{q-2}\psi_2 
	\end{matrix}\right)
\end{equation}
where  $q,\beta$ are suitable parameters.\\
Observe that, for $q=2$, the approximation in \eqref{Vexnostro} becomes
\[
(V_{\rm ex} \psi) = -C
\left(\begin{matrix}
\vspace{5pt}\psi_1 & \beta \psi_1\\
\beta \psi_2 & \psi_2
\end{matrix}\right)
\left(\begin{matrix}
\vspace{5pt}|\psi_1|^{2}\\
|\psi_2|^{2}
\end{matrix}\right)
=
-C
\left(\begin{matrix}
\vspace{5pt}(|\psi_1|^{2} +\beta |\psi_2|^{2})\psi_1\\
(\beta |\psi_1|^{2}+|\psi_2|^{2})\psi_2 
\end{matrix}\right),
\]
that is similar to the one applied by Slater in \eqref{VexSlater}, with a different power of the electron density function which is also {\em perturbed} by the parameter $\beta$.\\
Considering $\psi_1$ and $\psi_2$ real functions, renaming them as 
$u,v$, and taking, for simplicity, $C=1$, we get
\begin{equation}
	\label{system}\tag{$\mathcal{S}_{\lambda,\beta}$}
	\begin{cases}
		-\Delta u + u + \lambda \phi_{u,v} u = |u|^{2q-2} u + \beta |v|^q |u|^{q-2} u \medskip \\
		-\Delta v + v + \lambda\phi_{u,v} v = |v|^{2q-2} v + \beta |u|^q |v|^{q-2} v
	\end{cases}
	\text{in }\mathbb{R}^3,
\end{equation}
where 
$$\phi_{u,v} (x) := %\frac{1}{|\cdot |}*(u^{2} + v^{2})
\int_{\mathbb R^{3}} \frac{u^{2}(y) + v^{2}(y)}{|x-y |}dy\in D^{1,2}(\mathbb R^{3}),$$
where this last space is the closure of the test functions in the $L^{2}$-norm of the gradient.\\
Observe that $\phi_{u,v}$ is the unique solution of
$$-\Delta \phi=4\pi (u^{2} + v^{2})\quad  \text{in }\mathbb{R}^3.
$$
Thus, system \eqref{system} can be also seen as a Schr\"odinger-Poisson type system (see e.g. \cite{FMN}).

A particular case of system \eqref{system}, when $\lambda =0$, leads to the local weakly coupled nonlinear Schr\"odinger system 
\begin{equation}
	\label{NLS}
	\begin{cases}
		-\Delta u + u = |u|^{2q-2} u + \beta |v|^q |u|^{q-2} u \medskip \\
		-\Delta v + \omega^2 \,v = |v|^{2q-2} v + \beta |u|^q |v|^{q-2} v
	\end{cases}
	\text{in }\mathbb{R}^3,
\end{equation}
for $0 < \omega^2 \leq 1$, which has been intensively studied in the past fifteen years.
Applying variational methods, the first works are authored by Lin and Wei \cite{LW} and also by Ambrosetti and Colorado \cite{AC}, Maia, Montefusco, and Pellacci \cite{MMP}, Bartsch and Wang \cite{BW}, Sirakov \cite{S}, then followed by an extensive literature presenting investigations of different aspects and variations of this problem.\\
In fact this system is obtained when looking for solitary wave solutions of two coupled nonlinear Schr\"odin\-ger equations which model, for instance, binary mixtures of Bose-Einstein condensates or propagation of wave packets in nonlinear optics. In the present scenario, the self-interaction is attractive (self-focusing) and the interaction between the two components may be either attractive ($\beta > 0 $) or repulsive ($\beta <0$).
Many different and clever approaches have been provided in order to find ranges of parameter $\beta$ for which a positive (ground state) solution $(u,v)$ of the system is vectorial (namely having both nontrivial components) and so distinguish them from the semitrivial ones $(u,0)$ and $(0,v)$.
So far a remarkable amount of information has been made available on this matter,
including the proof in \cite{Mandel} of a threshold $\beta(\omega, q, n)$ for existence or nonexistence of vector ground states for problem \eqref{NLS} in $\mathbb{R}^n$.
% \textcolor{blue}{However the best range of the parameter $\beta$ for existence of such a vectorial solution has not been completely  settled yet for the problem in its most generality,
% when parameters multiply the power terms in $u$ and $v$.}

	The system above also arises as population dynamics are modelled and their associated reaction-diffusion equations in bounded or unbounded domains are studied using variational techniques; among many interesting works on this matter there are \cite{CTV02, CTV05, Soave} and references therein. When, for instance, an analysis is performed of the limiting case with respect to a parameter $\beta$ which describes interspecies competitions, going to plus or minus infinity, possible segregation states of two or more competing species are identified, leading to configurations where the populations occupy disjoint habitats.

In this paper we study the existence of solutions to problem
\eqref{system} in the unknowns $(u,v)\in \textrm H: = H^1(\mathbb{R}^3)\times H^1(\mathbb{R}^3)$.
%By a solution of \eqref{system} we mean a pair $(u,v)\in \textrm H: = H^1(\mathbb{R}^3)\times H^1(\mathbb{R}^3)$ such that
%for all $(w,z)\in \textrm H $ it holds
%\begin{align}
%\int \nabla u\nabla w +\int uw +\lambda \int \phi_{u,v} uw &= \int  |u|^{2q-2} uw + \beta\int |v|^q |u|^{q-2} uw, \label{eq:defsol1}\\
%\int \nabla v \nabla z +\int vz +\lambda \int \phi_{u,v} vz &= \int  |v|^{2q-2} vz + \beta\int |v|^q |u|^{q-2} vz. \label{eq:defsol2}
%\end{align}
%{\color{blue}Hereafter all the integrals are intended on the whole $\mathbb{R}^3$, unless otherwise specified.
%Notice that each term in \eqref{eq:defsol1} and \eqref{eq:defsol2} are finite since $\phi_{u,v}\in D^{1,2}(\mathbb R^{3})$, where
%this last space is the closure of the test functions in the $L^{2}-$norm of the gradient.}
In particular we are interested in nontrivial solutions, namely $(u,v) \in \textrm H\setminus \{0\}:=\textrm H\setminus\{(0,0)\}$.

Our approach in solving problem \eqref{system} is variational. Indeed a $C^{1}$ energy
functional in $\textrm H$ can be defined such that  its critical points give exactly the solutions
of our system. 

%We remark that the properties of the functional 
%will depend on the values of the parameters $\beta$ and $q$ in the nonlinearity.

However in order to deal with compactness issues, we will work
(except for the nonexistence result) in the radial setting and we will use the compact embedding of $H_{\textrm r}^1(\mathbb{R}^3)$ into $L^p(\mathbb{R}^3)$ for $p\in(2,6)$, see   e.g. \cite{BL,Str}.
Then the functional will be restricted to 
$\textrm H_{\textrm{r}} : = H_{\textrm{r}}^1(\mathbb{R}^3)\times H_{\textrm{r}}^1(\mathbb{R}^3)$
and the solutions will be found in  $\textrm H_{\textrm{r}}$.
The invariance of the functional under rotations and the Palais' Principle of Symmetric Criticality \cite{Palais} makes natural this constraint.
%Note that due to the invariance of the equations under rotations 
%and  the Palais' Principle of Symmetric Criticality \cite{Palais}, these solutions
%will satisfy 
%%the radial functions form a natural constraint: once we find a (radial) solution it will satisfy 
%\eqref{eq:defsol1} and \eqref{eq:defsol2} also for any  $(w,z)\in  \textrm H$. 

Actually we are interested in the existence of  {\sl ground state solutions}: with this terms we mean  radial solutions 
whose energy is minimal among all the other  radial ones.\\
	Such {\em definition} is motivated by the fact that for our system \eqref{system}, as well as for the corresponding scalar problem
	\begin{equation}\label{eq:ruiz}
		-\Delta u + u + \lambda \phi_{u} u = |u|^{2q-2} u\quad
		\text{in }\mathbb{R}^3, \quad \phi_{u}(x):= \int_{\mathbb{R}^3} \frac{u^{2}(y)}{|x-y|}dy,
	\end{equation}
	the classical Schwarz symmetrization or the polarization arguments (see \cite{Lieb,MVS}), that are enough to treat the nonlocal term, and so to prove the radial symmetry of the ground state solutions for the Choquard equation, are (or seem to be, respectively) not sufficient to guarantee the radial symmetry of our ground states. Indeed, in our case, as observed by Lieb in \cite{Lieb}, the  Riesz inequality implies that the  energy increases when we pass to the symetrized function.

%\textcolor{red}{Our first result concerns the nonexistence of any type of solutions.\\
%MI SEMBRA UN PO STRANO ``NESSUN TIPO DI SOLUZIONI''. VISTO CHE CI RIFERIAMO ALLE
%RADIALI, TANTO VALE MEGLIO DIRLO}

%To state the existence result,

In order to state our main result concerning  the existence of ground state solutions for $q\in (3/2, 3)$, their vectorial or semitrivial nature, and their  asymptotic behaviour  with respect to the parameter $\beta$,
let us first recall  that in \cite{RuizJFA}
it was proved that, for any $\lambda>0$, the equation \eqref{eq:ruiz}
%\begin{equation}\label{eq:ruiz}
%-\Delta u + u + \lambda \phi_{u} u = |u|^{2q-2} u\quad
%\text{in }\mathbb{R}^3, \quad \phi_{u}(x):= \int_{\mathbb{R}^3} \frac{u^{2}(y)}{|x-y|}dy,
%\end{equation}
possesses a radial ground state solution among all the radial solutions
which will be denoted hereafter with $\mathfrak w\in H^{1}_{\textrm r}(\mathbb R^{3})$.\\ Observe that, whenever a ground state of \eqref{system} is semitrivial, then, necessarily, it is of the type $(\mathfrak w,0)$ or $(0,\mathfrak w)$.

%Then we can state the result on existence  of ground states solutions and its asymptotic behaviour
% As usual for these type of systems, we speak of {\sl semitrivial } ground state
%when (only) one component is zero; in case they are both nontrivial we speak of  {\sl vectorial } ground state.

We have

\begin{Th}\label{th:gs}
Let $q\in (3/2, 3)$, $\lambda>0$, and $\beta\geq0$. Then \eqref{system} has a  radial  ground state solution $(\mathfrak u_{\beta},\mathfrak v_{\beta})\neq(0,0)$. Moreover: 
\begin{enumerate}[label=(\roman{*}), ref=\roman{*}]
\item\label{th:gsi} if $\beta=0$,  the  ground state solution  is semitrivial; \medskip
\item\label{th:gsiib} if $\beta\in (0,2^{q-1}-1)$ and $q\in [2,3)$, the ground state solution  is semitrivial;
\item \label{th:gsiiib} if $\beta\in(0,q-1)$ and $q\in (3/2,2)$, the ground state solution  is vectorial and
\[
\lim_{\beta\to 0^+} \operatorname{dist}_{\rm H} (\mathcal{G}_\beta,\mathcal{G}_0)=0
\]
where $\mathcal{G}_\beta:=\{(\mathfrak u_\beta , \mathfrak v_\beta)\in {\rm H}_{\rm r}: (\mathfrak u_\beta , \mathfrak v_\beta) \text{ is a ground state of \eqref{system}}\}$ and $\mathcal{G}_0:=\{(\mathfrak w,0),(0,\mathfrak w)\}$;
\item\label{th:gsiii} if 
\begin{equation}
	\label{betalarge}
	\beta \in 
	\begin{cases}
		[q-1,+\infty) & \text{ for }q\in (3/2,2),\\
		(2^{q-1}-1,+\infty) & \text{ for } q\in [2,3),
	\end{cases}
\end{equation}
the ground state $(\mathfrak u_{\beta}, \mathfrak v_{\beta})$ is
vectorial and
\begin{equation} \label{betainfinito}
\lim_{\beta\to+\infty} (\mathfrak u_\beta , \mathfrak v_\beta) =(0,0)
\text{ in }{\rm H}_{\rm r};
\end{equation}
\item\label{th:gsivb} if $\beta=2^{q-1}-1$ and $q\in [2,3)$, system \eqref{system}
admits both semitrivial and vectorial ground states.
\end{enumerate}
\end{Th}

Some remarks on our result are now in order.

The presence of the nonlocal Coulomb type coupling in \eqref{system} implies several difficulties with respect to system like \eqref{NLS}, in particular  for what concerns the semitrivial or vectorial nature of the ground states, which is, auctually, the main goal of the paper.\\
Indeed system \eqref{NLS} when $\beta=0$ or when we consider semitrivial solutions, reduces to 
single equation 
$$-\Delta u+u = |u|^{2q-2}u\quad\text{in }\mathbb R^{3}.$$
For such   equation, well known results have been obtained about uniqueness of the positive solution, its nondegenracy, its radial symmetry and exponential decay. These facts are used in the study of \eqref{NLS} (see \cite{MMP,Mandel}). \\
In our case, even for $\beta=0$ the system remains coupled in the nonlocal term.\\
Moreover, even if, for semitrivial solutions, system \eqref{system} reduces to a single equation, for such equation no result about uniqueness, nondegeneracy, and eventual symmetries of positive solution are known.\\
Finally, to deal with powers $q\in(3/2,3)$, following \cite{RuizJFA}, we use a rescaling (see \eqref{gamma}) which generates different behaviors 
of the terms in the functionals but that anyway allows us to project any nontrivial couple $(u,v)$ in a suitable manifold.
 Actually, for the simpler case $q\in(2,3)$, the usual projection on the Nehari manifold
is enough.\\
Nevertheless, our analysis shows that the nature of the ground states depends on the local nonlinearity.
Indeed our results are comparable with the ones in \cite{Mandel}, even if they
are obtained in a different way: we start from the existence of  ground states and, using
the maximum values of a suitable one variable function related to the local nonlinearity (see Lemma \ref{unionefacili}),
we estimate the ground state energy level and construct also  a particular family of ground states
(see Lemma \ref{Mbeta}), that, in the particular case $q=2$ and $\beta=1$, gives infinitely many ground states.

Additionally, due to the symmetry in $u$ and $v$ of \eqref{system}, it is easy to obtain nontrivial solutions with $u=v$ (see Remark \ref{rem2l}). For $\beta$ large enough, such solutions are ground states (Theorem \ref{A2}) and, for $\beta$ small, they are not (Theorem \ref{A1}).

%By {\em radial ground state solution} we mean that it minimize the functional in ${\rm H}_r$.
%It is important to observe that at this stage the ground state may be semitrivial, that is of type $(\mathfrak u, 0)$ or 
%$(0, \mathfrak v)$ with $\mathfrak u, \mathfrak v$ nonzero.
%\todo[inline]{For solutions we mean finite energy solutions.}
%\begin{Rem}\label{classical}
Finally, the solutions we find are  classical.
	%Any solution $(u,v)$ of \eqref{system} is classical. 
	Indeed, if $(u,v)\in {\rm H}_{\rm r} $, then $\phi_{u,v} \in W_{\rm loc}^{2,3}(\mathbb R^{3})$ and then it is
	$C^{0,\alpha}_{\rm loc}(\mathbb R^{3})$. But then by bootstrap arguments $u,v\in C^{2,\alpha}_{\rm loc}(\mathbb R^{3})$ which in turn implies $\phi_{u,v}\in  C^{2,\alpha}_{\rm loc}(\mathbb R^{3})$.
%\end{Rem}
Moreover, by the Maximum Principle, every nontrivial component of a solution can be assumed strictly positive
without loss of generality.

Of course our problem can be written using the equivalent complex notation $\psi:=u+iv$. Observe that, with such a notation,
\[
\int_{\mathbb R^{3}} \frac{u^{2}(y) + v^{2}(y)}{|x-y |}dy
=\int_{\mathbb R^{3}} \frac{|\psi (y)|^{2}}{|x-y |}dy,
\]
depending  only on $|\psi|$. For our scopes, especially in order to distinguish between semitrivial and vectorial ground states, in the analysis it should be necessary to use real and imaginary parts of $\psi$ and so we will proceed using the vectorial notation $(u,v)$.

%{\color{blue}Our approach should work also in the local case studied in \cite{MMP}.}

Additionally, we prove also the following nonexistence result.

%To avoid technical problems due to the fact that we are considering the whole $\mathbb{R}^3$, we work on the %subspace of radial functions ${\rm H}_r:=H_r^1(\mathbb{R}^3)\times H_r^1(\mathbb{R}^3)$ which, by the Palais %Principle of Symmetric Criticality \cite{Palais}, is a natural constraint.

%Let us start with the   nonexistence result. Notice it holds for any type  of solution: radial or not,
%ground state or not.

\begin{Th}\label{th:nonexistence}
	In ${\rm H} \cap (L^{2q}(\mathbb{R}^3)\times L^{2q}(\mathbb{R}^3))\cap(L_{\rm loc}^\infty(\mathbb{R}^3)\times L_{\rm loc}^\infty(\mathbb{R}^3))$, system \eqref{system} has only the trivial solution  if $q\geq 3$  and no solution with {\em fixed sign} if  $q\in [1/2,1]$.
	%System \eqref{system} has only the trivial solution in ${\rm H} \cap (L^{2q}(\mathbb{R}^3)\times L^{2q}(\mathbb{R}^3))$ for $q\in [1/2,1]\cup [3,+\infty[ $.
\end{Th}
Here, with {\em fixed sign solution}, we mean couples $(u,v)$ where each component is strictly positive or negative.

\medskip

The paper is organized as follows.\\
In Section \ref{sec2} we present few preliminaries in order to prove our results.
In particular we recall some  results in \cite{RuizJFA} 
that will be used to compare the ground
state level of our functional (for example to study the asymptotic behaviour). 
We give also the variational setting for
our problem.\\
In Section  \ref{sec:EXNONEX} we prove the nonexistence result, Theorem \ref{th:nonexistence},
which is based on a Pohozaev identity associated to the problem.
Then we give also the proof of the existence of a nontrivial ground state
in Theorem \ref{th:gs}.\\
Then item (\ref{th:gsi}) is proved in Section \ref{sec:bzero}, (\ref{th:gsiib}) and  (\ref{th:gsiiib})
 are proved in Section \ref{sec:POSITIVESMALL}, (\ref{th:gsiii}) is proved in Section
 \ref{sec:blarge}, and (\ref{th:gsivb}) in Section \ref{sec:7}. \\
We complete Section \ref{sec:POSITIVESMALL} and Section \ref{sec:blarge}
showing that some particular solutions arising from the study of the single equation 
(see Remark \ref{rem2l}) are
or not ground states (see Theorem \ref{A1} and Theorem \ref{A2}, respectively).

\subsection*{Notations}
\begin{itemize}
%	\item $\textrm{H}=H^{1}(\mathbb R^{3})\times H^{1}(\mathbb R^{3}) $ and 
%	of course, $\textrm{H}\setminus\{0\}= H^{1}(\mathbb R^{3})\times H^{1}(\mathbb R^{3})\setminus\{(0,0)\}.$
%	\item given $(u,v)\in \textrm{H},$ we set $\phi_{u,v} = |\cdot|^{-1}*(u^{2}+v^{2})\in D^{1,2}(\mathbb R^{3})$ the 
%unique solution
%	of $$-\Delta\phi = 4\pi (u^{2}+v^{2}).$$
	\item Unless otherwise stated, integrals  will always be considered
	on the whole $\mathbb R^{3}$ with the Lebesgue measure.
	\item We denote with $\|\cdot\|$ the norm in $H^{1}(\mathbb R^{3})$ and with 
	$\|\cdot\|_{p}$ the standard $L^{p}-$ norm.
	\item We denote with $\varepsilon_{n}$ a generic  sequence which vanishes as $n$ tends to infinity and with
	$C$ a suitable positive constant 
	%whose value does not really matter and 
	that can vary from line to line.
%\item If $X$ is a function space,  we denote by $X_{r}$ the subspace of $X$ in which the elements are radial functions in %$X$.
\end{itemize}
Other notations will be introduced whenever needed.	

%$$\Phi: (u,v)\in\textrm{ H}_{r}\longmapsto \phi_{u,v} \left( = \frac{1}{|\cdot |}*(u^{2} + v^{2})\right)\in D_{r}^{1,2}(\mathbb R^{3}).$$

\section{Preliminary results }\label{sec2}

%\todo[inline]{In questa sezione potremmo mettere:\\
%(i) proprit\`a della $\phi$ (alcume stanno sopra nell'inroduzione, io aggiungerei anche $\|\phi\|\leq c (\|u\|+\|v\|)$);\\
%(ii) guardare qualcosa sulla regolarit\`a del funzionale $I$.}

%-----------------------------------
%%%  SPOSTATO NELL'INTRO
%In all the paper we will work in the radial setting so the compact embedding of 
 %$H_r^1(\mathbb{R}^3)$ into $L^p(\mathbb{R}^3)$ for $p\in(2,6)$ will be always  tacitly used (see 
 %e.g. \cite{BL,Str}). 
 %Let us observe that due to the invariance of the equation under rotations,
%and  the Palais' Principle of Symmetric Criticality, 
%the radial functions form a natural constraint: once we find a (radial) solution it will satisfy  \eqref{eq:defsol1}
%and \eqref{eq:defsol2} also
%for any  $(w,z)\in  \textrm H$. 
%-----------------------------------

\medskip

In order to prove our results, let us first recall some facts 
%from \cite{RuizJFA}
about \eqref{eq:ruiz}.
%the equation
%	\begin{equation}\label{RuizJFA}
%-\Delta u + u + \lambda \phi_u u = |u|^{2q-2} u, \quad  	u\in H^{1}_{\textrm r}(\mathbb R^{3}),
%\end{equation}
%where
%$ \phi_u = \displaystyle  \frac{1}{|\cdot|}*u^2$ is the unique solution of
%\[
%-\Delta \phi = 4\pi u^2 \quad 
%\text{ in } \mathbb{R}^3.
%\]
%\todo[inline]{Dire da qualche parte che il radial setting \`e un vincolo naturale.}
%	\todo[inline]{To reach our goal we need to compare/consider the one variable case (dire bene).}
%For every $u\in H^1(\mathbb{R}^3)$, let $\phi_u = \displaystyle  \frac{1}{|\cdot|}*u^2$ be the unique solution of
%	\[
%	-\Delta \phi = 4\pi u^2
%	\text{ in } \mathbb{R}^3.
%	\]
%	and consider the equation
%	\begin{equation}\label{RuizJFA}
%	-\Delta u + u + \lambda \phi_u u = |u|^{2q-2} u \quad  	\text{ in } \mathbb{R}^3.
%	\end{equation}
%	Let  $\mathfrak w\not\equiv 0$ be a {\em radial ground state} solution of \eqref{eq:ruiz}
%	between the radial functions in $H^1(\mathbb{R}^3)$ 
	In \cite{RuizJFA} it was proved that for any $\lambda>0$ and $q\in (3/2, 3)$, equation 	\eqref{eq:ruiz} has a {\em radial ground state} solution
	$\mathfrak w\in H^{1}_{\textrm r}(\mathbb R^{3})\setminus\{ 0\}$.
	It is found as a minimizer of the $C^1$-functional
	\[
	\mathcal{I}_{\lambda,0}(u):=
	\frac{1}{2} \|\nabla u \|_2^2
	+ \frac{1}{2} \| u \|_2^2
	+ \frac{\lambda}{4}  \int  u^2 \phi_u
	-\frac{1}{2q}  \|u\|_{2q}^{2q}, \quad u\in H^{1}_{\textrm r}(\mathbb R^{3})
	\]
	on the constraint
	\begin{equation}\label{eq:NRuiz}
	\mathcal{N}^\lambda := \left\{u\in H_{\textrm r}^1(\mathbb{R}^3): \mathcal{J}_{\lambda,0} (u)=0\right\},	
	\end{equation}
	where 
	\[
	\mathcal{J}_{\lambda,0} (u):= \frac{3}{2}  \|\nabla u \|_2^2
	+ \frac{1}{2} \| u \|_2^2
	+ \frac{3}{4} \lambda \int \phi_u u^2
	-\frac{4q-3}{2q}  \|u\|_{2q}^{2q}.
	\]
%	To the best of our knowledge, this constraint has been introduced for the first time in \cite{RuizJFA}
%	and  
	The set $\mathcal N^{\lambda}$ is obtained as a {\em linear combination} of the Nehari identity
	\[
	\| \nabla u\|_2^2
	+ \|u\|_2^2
	+\lambda \int \phi_u u^2
	-\|u\|_{2q}^{2q}=0
	\]
	and the Pohozaev identity
	\[
	\frac{1}{2}\| \nabla u\|_2^2
	+ \frac{3}{2}\|u\|_2^2
	+\frac{5}{4}\lambda \int \phi_u u^2
	-\frac{3}{2q}\|u\|_{2q}^{2q}=0.
	\]
Given $u\not\equiv 0$, consider the path
	\begin{equation}
\label{zeta}
\zeta_{u}(t):=t^2 u(t\cdot), \quad   t\geq 0.
\end{equation}
Note that
\begin{equation*}
\label{Ionevart}
\mathcal{I}_{\lambda,0}(\zeta_u(t))=
\frac{t^3}{2} \|\nabla u \|_2^2
+ \frac{t}{2} \| u \|_2^2
+ \frac{\lambda}{4} t^3 \int  u^2 \phi_u
-\frac{t^{4q-3}}{2q}  \|u\|_{2q}^{2q},
\end{equation*}
\begin{equation}
\label{Jonevart}
\mathcal{J}_{\lambda,0} (\zeta_u(t))
= \frac{3}{2} t^3 \|\nabla u \|_2^2
+ \frac{1}{2} t \| u \|_2^2
+ \frac{3}{4} \lambda t^3 \int \phi_u u^2
-\frac{4q-3}{2q} t^{4q-3} \|u\|_{2q}^{2q},
\end{equation}
and $t\mapsto \mathcal I_{\lambda,0}(\zeta_{u}(t))$ has a unique critical point,
 denoted with $t_{u}>0$  corresponding to its maximum.
 The elements of $\mathcal N^{\lambda}$ are then all of type $\zeta_{u}(t_{u})$ due to the fact that
 	\[
 \mathcal{J}_{\lambda,0} (\zeta_{u}(t))
 =\frac{d}{dt}\mathcal{I}_{\lambda,0}(\zeta_{u}(t)).
 \]
 In particular $u\in \mathcal N^{\lambda}$ if and only if $t_{u}=1$ and
%obtained by computing the functional on the path
%	Moreover, if $u\in H_r^1(\mathbb{R}^3)\setminus\{0\}$, through the path
%	\begin{equation}
%	\label{zeta}
%	\zeta(t):=t^2 u(t\cdot), \quad  u\not\equiv 0,  t\geq 0,
%	\end{equation}
%	In the following, if no ambiguity occurs, we will write simply $\zeta$ for the path associated to $u$.
%	 the function $u$ is clear from the context, we will simply write
%	 the dependence of the path in \eqref{zeta} 
%	by the function $u$ by writing $\zeta_{u}$.
%	and it is easily seen that there is a unique $t_{u}>0$ such that $t\mapsto \mathcal I_{\lambda,0}(\zeta(t)) $ 
%	achieves its maximum.
%	Since  it holds
%	\[
%	\mathcal{J}_{\lambda,0} (\zeta(t))
%	=\frac{d}{dt}\mathcal{I}_{\lambda,0}(\zeta(t))
%	\]
%	we see that
%	$\mathcal J_{\lambda,0 } (\zeta (t_{u}))  =0$ meaning that  $\zeta(t_{u})\in \mathcal N^{\lambda}$.
%	emphasizing the dependence on $u$, if necessary, we can project $u$ into $\mathcal{N}^\lambda$, namely there exists a unique 
%$t_u>0$ such that $\zeta(t_u)\in \mathcal{N}^\lambda$ and an easy calculation shows that
%	and
%	\begin{equation}
%	\label{Jonevart}
%	\mathcal{J}_{\lambda,0} (\zeta(t))
%	= \frac{3}{2} t^3 \|\nabla u \|_2^2
%	+ \frac{1}{2} t \| u \|_2^2
%	+ \frac{3}{4} \lambda t^3 \int \phi_u u^2
%	-\frac{4q-3}{2q} t^{4q-3} \|u\|_{2q}^{2q}.
%	\end{equation}$$
%	Hence, since
%	\[
%	\mathcal{J}_{\lambda,0} (\zeta(t))
%	=\frac{d}{dt}\mathcal{I}_{\lambda,0}(\zeta(t))
%	\]
%	and by (\ref{fa}) in Lemma \ref{lem:calculus}, we have that
%	\todo[inline]{La prima uguaglianza vale per tutte le proiezioni. Scriverlo da qualche parte.}
then %by  \cite{RuizJFA}, we know that
%	\[
%	\mathcal{I}_{\lambda,0}(\mathfrak{w})
%	=\mathcal{I}_{\lambda,0}(\zeta_{\mathfrak{w}}(1))
%	=\max_{t>0} \mathcal{I}_{\lambda,0}(\zeta_{\mathfrak{w}}(t))>0
%	\]
	\begin{equation}  
	\label{gslRuiz}
0<	\mathcal{I}_{\lambda,0}(\mathfrak{w})
	= \inf_{u\in\mathcal{N}^\lambda} \mathcal{I}_{\lambda,0} (u)
%	= \mathcal{I}_{\lambda,0}(\zeta_{\mathfrak{w}}(1))
= \inf_{u\in H_{\textrm r}^1(\mathbb{R}^3) \setminus\{0\}} \mathcal I_{\lambda,0}(\zeta_{u}(t_{u}))
	= \inf_{u\in H_{\textrm r}^1(\mathbb{R}^3) \setminus\{0\}} \max_{t>0} \mathcal{I}_{\lambda,0}(\zeta_u(t)).
	\end{equation}
	
	\begin{Rem}
	Of course  $(\mathfrak{w},0)$ and $(0,\mathfrak{w})$ are semitrivial solutions of our system \eqref{system} for any $\beta$ and so,  since $I_{\lambda,\beta}(u,0) =I_{\lambda,\beta}(0,u) =  \mathcal I_{\lambda,0}(u)$, 	
	they are necessarily ground state  whenever the ground state  is semitrivial.
	\end{Rem}

%	\todo{Mettere dove varia $\beta$.}).\\
For future reference we set
\begin{equation*}\label{eq:n}
\mathfrak n := \mathcal{I}_{\lambda,0}(\mathfrak{w}).
\end{equation*}

Moreover, the same arguments of \cite{RuizJFA} can be repeated for the equation
\begin{equation}
\label{2lambdabeta}
-\Delta u + u + 2\lambda \phi_{u} u = (1+\beta)|u|^{2q-2} u \quad \text{ in } \mathbb R^{3}
\end{equation}
where $\beta\geq0$, leading to the existence of a ground state solution 
$\mathfrak z_\beta$ that minimizes the functional
\[
\mathcal I_{2\lambda, \beta}{(u)}:=\mathcal I_{2\lambda, 0}(u) -\frac{\beta}{2q} \|u\|_{2q}^{2q} =
%\textcolor{red}{I_{2\lambda,\beta}(u)=}
\frac{1}{2} \|\nabla u \|_2^2
+ \frac{1}{2} \| u \|_2^2
+ \frac{\lambda}{2}  \int \phi_{u} u^2
-\frac{1+\beta}{2q} \|u\|_{2q}^{2q}
\]
on  the set of  $u\in H^{1}_{\textrm{r}}(\mathbb R^{3})$
satisfying
\[
\mathcal J_{2\lambda, \beta} (u):=
%\textcolor{red}{\mathcal J_{2\lambda, 0}(u) -\frac{4q-3}{2q}\beta \|u\|_{2q}^{2q} =}
\frac{3}{2}  \|\nabla u \|_2^2
+ \frac{1}{2} \| u \|_2^2
+ \frac{3}{2} \lambda \int \phi_{u} u^2
-\frac{4q-3}{2q}  (1+\beta) \|u\|_{2q}^{2q}=0.
\]

Coming back to our system \eqref{system}, observe that
it can be  written as
\begin{equation}\label{eq:system3}
\begin{cases}
-\Delta u + u + \lambda \phi u = |u|^{2q-2} u + \beta |v|^q |u|^{q-2} u \medskip \\
-\Delta v + v + \lambda \phi v = |v|^{2q-2} v + \beta |u|^q |v|^{q-2} v \medskip \\
-\Delta \phi = 4\pi(u^{2}+v^{2})
\end{cases}
\text{in }\mathbb{R}^3.
\end{equation}
Moreover
\begin{equation}\label{lastint}
\int_{\mathbb R^{3}} |\nabla \phi_{u,v}|^{2} = 4\pi \int_{\mathbb R^{3}}  (u^{2}+v^{2})\phi_{u,v}
\end{equation}
from which  the estimate follows
 $$\| \nabla \phi_{u,v}\|_2 \le C \left( \|u\|^2 + \|v\|^2\right).$$

It is standard to see that the weak solutions of \eqref{system} %(see \eqref{eq:defsol1}, \eqref{eq:defsol2})
are
characterised as the critical points of the $C^1$ functional
defined on $\textrm H$
\begin{align*}
I_{\lambda,\beta}(u,v)
&=
\frac{1}{2} \|\nabla u \|_2^2
+ \frac{1}{2} \| u \|_2^2
+ \frac{1}{2} \|\nabla v \|_2^2
+ \frac{1}{2} \| v \|_2^2
+ \frac{\lambda}{4} \int (u^2+v^2)\phi_{u,v}  \\
&\quad
-\frac{1}{2q}(\|u\|_{2q}^{2q}+\|v\|_{2q}^{2q})
-\frac{\beta}{q}\int |u|^q |v|^q.
\end{align*}

\begin{Rem}\label{rem2l}
	Observe that, for every $\beta\geq0$ and $u\in H^1(\mathbb{R}^3)$,
	\begin{equation}
	\label{i2i}
	I_{\lambda,\beta}(u,u) =2 \mathcal I_{2\lambda,\beta}(u)
	\end{equation}
	and, $(u,u)$ is a solution of \eqref{system} if and only if $u$ is a solution of \eqref{2lambdabeta}.
\end{Rem}

If $(u,v)$ is  a solution of \eqref{system}, 
 multiplying the first equation of the system by $u$ and the second one by $v$ we  see that
%have that any solution 
$(u,v)\in {\rm H}$  satisfies the Nehari type identities
\begin{align}
\|\nabla u \|_2^2
+ \| u \|_2^2 
+ \lambda  \int u^2 \phi_{u,v} 
&=
\|u\|_{2q}^{2q}
+\beta\int |u|^q |v|^q,\label{N1}\\
\|\nabla v \|_2^2
+ \| v \|_2^2 
+ \lambda  \int  v^2 \phi_{u,v}
&=
\|v\|_{2q}^{2q}
+\beta\int |u|^q |v|^q. \label{N2}
\end{align}

Given $(u,v)\in\textrm{H}\setminus\{0\}$,
we denote with
$\gamma_{u,v}:\mathbb [0,+\infty[\to \textrm{H}$ the curve
\begin{equation}
\label{gamma}
\gamma_{u,v} (t):= (t^2 u(t\cdot), t^2 v(t\cdot )).
\end{equation}
By a simple calculation we have that
\begin{equation*}\label{eq:Igamma}
\begin{split}
I_{\lambda, \beta}(\gamma_{u,v}(t))
&=
\frac{t^3}{2} (\|\nabla u \|_2^2 + \|\nabla v \|_2^2)
+ \frac{t}{2} (\| u \|_2^2 + \| v \|_2^2)
+ \frac{\lambda}{4} t^3 \int (u^2+v^2) \phi_{u,v} \\
&\quad
- \frac{t^{4q-3}}{2q}\left(
\|u\|_{2q}^{2q}+\|v\|_{2q}^{2q}
+2\beta\int |u|^q |v|^q
\right),
\end{split}
\end{equation*}
which will be useful in our arguments.

For future developments, we need the following results.

	\begin{lemma}\label{lem:calculus}
		Let $\mu,\nu,\sigma>0$, $p>3$, and consider the function $\mathfrak{f}_{\sigma}(t):=\mu t+\nu t^{3} - \sigma t^{p}$. Then
		\begin{enumerate}[label=(\alph{*}), ref=\alph{*}]
			\item \label{fa} $\mathfrak{f}_\sigma$ has a unique critical point $\mathfrak{t}_\sigma >0$ which corresponds to its maximum and there exists a unique $\mathfrak{T}_\sigma>\mathfrak{t}_\sigma$ such that $\mathfrak{f}_\sigma (\mathfrak{T}_\sigma)=0$;
			\item \label{fb} $\displaystyle\lim_{\sigma \to +\infty}  \mathfrak{f}_\sigma(\mathfrak{t}_\sigma) = 0$;
			%Given $c>0$ for $C$ sufficiently large it is $f(t)= At+Bt^{3} -Ct^{p} <c $ for all $t\geq0.$
			\item \label{fc} $\displaystyle \lim_{\sigma\to+\infty}  \mathfrak{T}_\sigma=0$ and $\displaystyle \mu=\lim_{\sigma\to+\infty} \sigma \mathfrak{T}_\sigma^{p-1}$. 
		\end{enumerate}
		% exists a $C_{*}>0$ such that for any $C'\geq C_{*}$ it holds
		%$$\max_{t\geq0} \left\{At+Bt^{3}-C' t^{p}\right\}\leq c.$$
	\end{lemma}
	\begin{proof}
		Property (\ref{fa}) is essentially \cite[Lemma 3.3]{RuizJFA} and is trivial.\\
		Let us prove (\ref{fb}). Since $p>3$, then, necessarily, $\mathfrak{t}_\sigma\to 0$ as $\sigma \to +\infty$. Indeed, if there exists $\bar{\mathfrak{t}}>0$ and a divergent sequence $\{\sigma_n\}$
		%$\sigma_n$>n$
		such that $\mathfrak{t}_{\sigma_n}>\bar{\mathfrak{t}}$, then
		\[
		\mu =p\sigma_n \mathfrak{t}_{\sigma_n}^{p-1}-3\nu\mathfrak{t}_{\sigma_n}^{2}
		=\mathfrak{t}_{\sigma_n}^{2} (p\sigma_n\mathfrak{t}_{\sigma_n}^{p-3}-3\nu)
		> \bar{\mathfrak{t}}^{2} (p\sigma_n\bar{\mathfrak{t}}^{p-3}-3\nu)
		\to +\infty
		\]
		giving a contradiction.
%		{\color{red}Moreover
%		\begin{equation}\label{eq:comportamento}
%		\lim_{\sigma\to+\infty}\sigma \mathfrak t_{\sigma}^{p-1}
%		=\lim_{\sigma\to+\infty} \left( \frac{\mu}{p} + \frac{3\nu}{p}\mathfrak t_{\sigma}^{2}\right)
%		=\frac{\mu}{p}.
%		\end{equation}
%		Thus
%		$$\mathfrak{f}_{\sigma}(\mathfrak t_{\sigma}) = \mathfrak t_{\sigma} [(p-1)\sigma \mathfrak t_{\sigma}^{p-1} -2\nu\mathfrak t_{\sigma}^{2}] 
%		\to 0  \text{ as } \sigma \to+\infty.$$	}
	Thus
		$$\mathfrak{f}_{\sigma}(\mathfrak t_{\sigma}) = \mathfrak t_{\sigma}
		\Big(
		\frac{p-1}{p}\mu + \frac{p-3}{p}\nu t_\sigma^2\Big)
		\to 0  \text{ as } \sigma \to+\infty.$$
		As for  (\ref{fc}), since $\mathfrak{T}_{\sigma}$ satisfies
		\begin{equation}\label{eq:T}
			\mu = \sigma \mathfrak{T}_{\sigma}^{p-1}- \nu \mathfrak{T}_{\sigma}^{2}
		\end{equation}
		we deduce,
%	{\color{red}that $T_{C}$ has to be bounded as $C\to+\infty$.}
		as in item (\ref{fb}),
%		{\color{red}we deduce}
		that
		$\mathfrak{T}_{\sigma} \to 0$ as $\sigma\to+\infty$ and so, coming back to \eqref{eq:T}, we conclude.
	\end{proof}

Now we state a fundamental tool that will allow us to distinguish the nature of the ground states pairs, identifying whether they are semitrivial or vectorial (see also Remark \ref{rem:generale}). Its proof is quite technical and involves simple analytical arguments. So we postpone it in the Appendix \ref{AppLemma24}.

\begin{Lem}
\label{unionefacili}
Let $\mathfrak{h}_\beta (y):= y^{q}+(1-y)^{q}+2\beta y ^{q/2} (1-y)^{q/2}$, $y\in[0,1]$, $\beta\geq 0$ and $q>1$.
\begin{enumerate}[label=(\roman{*}), ref=\roman{*}]
	\item \label{facilei}If $\beta=0$, then $\mathfrak{h}_0(y) \leq 1$ and the equality holds only in the endpoints $y=0,1$.
	
	\item \label{facileiib}  If $q\in(3/2,2)$, then, for any fixed $\beta>0$, there exists a unique $y_\beta\in(0,1/2]$ such that $\mathfrak{h}_\beta(y_\beta)=\mathfrak{h}_\beta(1-y_\beta)=\max_{y\in[0,1]}\mathfrak{h}_\beta(y)>1$  and $\displaystyle \lim_{\beta\to 0^{+}} y_{\beta} 
		=
		%1-\lim_{\beta\to 0^{+}} \theta_{\beta}=
		0$. Moreover $y_\beta=1/2$ if and only if $\beta\geq q-1$.

	\item \label{facileiiib} If $q\in[2,3)$, then:
	\begin{enumerate}[label=(\alph{*}), ref=\alph{*}]
		\item\label{aiiib} for $\beta\in (0,2^{q-1}-1)$, $\mathfrak{h}_\beta(y)\leq 1$ and the equality holds just in the endpoints $y_\beta=0,1$;
		\item\label{biiib} for $\beta=2^{q-1}-1$, $\mathfrak{h}_\beta(y)\leq 1$ and, in particular, 
		$$
		\mathfrak h_{\beta}(y)= 1\ \ \text{ in }  \ \ 
		\begin{cases}
		0,1/2,1 & \mbox{ if } q\in (2,3)\\
		[0,1]&\mbox{ if } q=2;
		\end{cases}
		$$
		\item\label{ciiib} $\beta > 2^{q-1}-1$, then $\mathfrak h_{\beta}$ achieves its unique global maximum on $y_\beta=1/2$ and $\mathfrak{h}_\beta(1/2)> 1$.
	\end{enumerate}
	
%	\item \label{facileii} {\color{red}If $q \geq 2$, then, for $\beta>0$ sufficiently small, $\mathfrak{h}_\beta(y)\leq 1$ and the equality holds just in the endpoints $y=0,1$.}
%	
%	\item \label{facileiv}{ \color{red}If $q\in(3/2,2)$, then, for any fixed $\beta>0$ sufficiently small, $\mathfrak{h}_{\beta}$ has two maximum points $y_{\beta}\in(0,1/2)$ 
%	and $1-y_{\beta}\in(1/2,1)$ satisfying $\mathfrak{h}_{\beta}(y_{\beta}) = \mathfrak{h}_{\beta}( 1-y_{\beta})>1$ and $\displaystyle \lim_{\beta\to 0^{+}} y_{\beta} 
%	=
%	%1-\lim_{\beta\to 0^{+}} \theta_{\beta}=
%	0$.}
%	
%	\item\label{facilev}{\color{red}If $\beta$ is sufficiently large, then $\mathfrak h_{\beta}$ achieves its unique absolute maximum on $y=1/2$.}
\end{enumerate}
\end{Lem}

\section{Existence and nonexistence results}\label{sec:EXNONEX}

In this section we prove the nonexistence result stated in  Theorem 
\ref{th:nonexistence} and the existence of a nontrivial  radial ground state of \eqref{system}, i.e. the first part of Theorem \ref{th:gs}.

\subsection{A Pohozaev identity and the nonexistence result}
As it is usual for elliptic equations, 
%\textcolor{red}{without any symmetry assumption,} 
the solutions satisfy a suitable identity called {\em Pohozaev identity}.
It can be obtained, at least formally, by the relation
\[
{\frac{d}{dt} I_{\lambda,\beta}(u_t,v_t)}\Big|_{t=1}=0 \ \ \text{ where } \ \ u_t(x):=u(x/t).
\]
In the next lemma we get it rigorously. The proof is indeed standard, however we revise  the argument for the sake of completeness. 
In what follows	$B_{R}$  stands for the ball centred in $0\in \mathbb R^{3}$ and radius $R>0$.

\begin{lemma}\label{lem:Pohozaev}
If $(u,v,\phi)$   is a  solution of \eqref{eq:system3} with $(u,v)\in {\rm H} \cap (L^{2q}(\mathbb{R}^3)\times L^{2q}(\mathbb{R}^3))\cap(L_{\rm loc}^\infty(\mathbb{R}^3)\times L_{\rm loc}^\infty(\mathbb{R}^3))$, with {\em fixed sign} if $q\in[1/2,1]$, then it satisfies the Pohozaev identity
\begin{equation}
\label{Pohosyst}
\begin{split}
&\frac{1}{2} (\|\nabla u \|_2^2 + \|\nabla v \|_2^2)
+ \frac{3}{2} (\| u \|_2^2 + \| v \|_2^2)
+ \frac{5}{4}\lambda  \int (u^2+v^2) \phi 
=\frac{3}{2q}\left(\|u\|_{2q}^{2q}+\|v\|_{2q}^{2q}
+2\beta\int |u|^q |v|^q\right).
\end{split}
\end{equation}
\end{lemma}
\begin{proof}
Let $(u,v,\phi)$ be a solution of \eqref{eq:system3}. If $q\in[1/2,1]$, without loss of generality, we can assume $u,v>0$.\\
Preliminarily we recall (see also \cite[Proposition 2.1]{BL} and \cite[Lemma 3.1]{DM}) that  for any $R>0$
\begin{align}
&\int_{B_{R}} - \Delta u \, x\cdot \nabla u = -\frac{1}{2} \int_{B_{R}} |\nabla u|^{2} -\frac1R\int_{\partial B_{R}} |x\cdot \nabla u|^{2} +\frac{R}{2} \int_{\partial B_{R}} |\nabla u|^{2}, \label{P1}\\
&\int_{B_{R}} \phi u \, x\cdot \nabla u = -\frac{1}{2} \int_{B_{R}} u^{2} \, x\cdot \nabla \phi -\frac{3}{2}\int_{B_{R}} \phi u^{2}+\frac{R}{2}\int_{\partial B_{R}} \phi u^{2}, \label{P2}\\
&\int_{B_{R}} g(u) \, x\cdot \nabla u = -3\int_{B_{R}} G(u) +R\int_{\partial B_{R}} G(u),\label{P3}\\
&\int_{B_R} |u|^{q-2} u |v|^q  x \cdot \nabla u + |u|^q |v|^{q-2} v x \cdot \nabla v
= -\frac{3}{q} \int_{B_R} |u|^q |v|^q + \frac{R}{q} \int_{\partial B_{R}} |u|^q |v|^q, \label{P4}
%&\int_{B_R} \partial_u D (u,v) \, x \cdot \nabla u + \partial_v D (u,v) \, x \cdot \nabla v
%= -3 \int_{B_R} D(u,v) + R \int_{\partial B_{R}} D(u,v), \label{P4}
\end{align}
where $g:\mathbb R\to \mathbb R$ is a continuous function with primitive $\displaystyle G(s) = \int_{0}^{s} g(\tau)d\tau$.\\
Observe that all the previous integrals make sense due to the regularity of $u,v,\phi$. 
In particular, since $u \in L_{\rm loc}^\infty (\mathbb{R})$, then $\Delta u\in W_{\rm loc}^{2,p} (\mathbb{R}^3)$ for all $p\geq 1$, the integral in the left hand side of \eqref{P1} is well defined.\\
Then, multiplying  the first equation in \eqref{eq:system3} by $x\cdot\nabla u$, integrating on $B_{R}$, and taking into account \eqref{P1}, \eqref{P2}, \eqref{P3}, \eqref{P4}, we get
\begin{equation}\label{eq:Poh1}
\begin{split}
&\frac{1}{2} \int_{B_{R}}|\nabla u|^{2}
+ \frac{3}{2}\int_{B_{R}}u^{2}
+\frac{\lambda}{2}\int_{B_{R}} u^{2} \, x\cdot\nabla \phi
+\frac{3\lambda}{2}\int_{B_{R}} \phi u^{2} 
-\frac{3}{2q} \int_{B_{R}} |u|^{2q}\\
&\quad
+\beta\int_{B_{R}} |v|^{q}|u|^{q-2}u \, x\cdot \nabla u 
=
-\frac1R \int_{\partial B_{R}}|x\cdot \nabla u|^{2}
+\frac{R}{2}\int_{\partial B_{R}} |\nabla u|^{2}
+\frac{R}{2} \int_{\partial B_{R}}u^{2}\\
&\quad
+\frac{\lambda R}{2}\int_{\partial B_{R}}\phi u^{2}
- \frac{R}{2q}\int_{\partial B_{R}} |u|^{2q}.
\end{split}
\end{equation}
In a similar way, from the second equation  in \eqref{eq:system3} we infer
\begin{equation}
\label{eq:Poh2}
\begin{split}
&\frac{1}{2} \int_{B_{R}}|\nabla v|^{2}
+ \frac{3}{2}\int_{B_{R}}v^{2}
+\frac{\lambda}{2}\int_{B_{R}} v^{2} \, x \cdot \nabla \phi 
+\frac{3\lambda}{2}\int_{B_{R}} \phi v^{2}
-\frac{3}{2q} \int_{B_{R}} |v|^{2q} \\
&\quad
+\beta\int_{B_{R}} |u|^{q}|v|^{q-2}v \, x\cdot \nabla v
= 
-  \frac1R \int_{\partial B_{R}}|x\cdot \nabla v|^{2}
+\frac{R}{2} \int_{\partial B_{R}} |\nabla v|^{2}
+\frac{R}{2} \int_{\partial B_{R}} v^{2}\\
&\quad
+\frac{\lambda R}{2} \int_{\partial B_{R}} \phi v^{2}
- \frac{R}{2q} \int_{\partial B_{R}} |v|^{2q}
\end{split}
\end{equation}
and, from the third one, multiplying by $x\cdot\nabla\phi$, we deduce
\begin{equation}\label{eq:Poh3}
\frac{1}{2}\int_{B_{R}}|\nabla \phi|^{2}
+4\pi \int_{B_{R}}  (u^{2}+v^{2}) \, x\cdot\nabla \phi
= -\frac1R \int_{\partial B_{R}}|x\cdot\nabla \phi|^{2}
+\frac{R}{2} \int_{\partial B_{R}}|\nabla \phi|^{2}.
\end{equation}
Then, summing up \eqref{eq:Poh1} and \eqref{eq:Poh2}, taking into account \eqref{eq:Poh3} and \eqref{P4}
%and that
% \begin{eqnarray*}
% \int_{B_{R}}\left( |u|^{q}|v|^{q-2}v x\cdot \nabla v + 
%|v|^{q}|u|^{q-2}u x\cdot \nabla u\right)&=& -\frac{3}{q}\int_{B_{R}}|u|^{q}|v|^{q}
%+\frac{1}{q}\int_{B_{R}}\text{div}\left( x |u|^{q}|v|^{q}\right) \\
%&=& 
%-\frac{3 }{q} \int_{ B_{R}} |u|^{q} |v|^{q} +\frac Rq\int_{\partial B_{R}} |u|^{q}|v|^{q}
% \end{eqnarray*}
we arrive at
\begin{align*}
&\frac12\int_{B_{R}}(|\nabla u|^{2} + |\nabla v|^{2})
+ \frac32\int_{B_{R}}(u^{2}+  v^{2})
-\frac{\lambda}{16\pi}\int_{B_{R}}|\nabla \phi|^{2}
+\frac{3\lambda}{2} \int_{B_{R}}(u^{2}+v^{2})\phi\\
&\quad
-\frac{3}{2q}\int_{B_{R}} (|u|^{2q}+|v|^{2q} 
+2\beta |u|^{q}|v|^{q})
=
-\frac1R \int_{\partial B_{R}}(|x\cdot \nabla u|^{2}+|x\cdot \nabla v|^{2})
+\frac{R}{2}\int_{\partial B_{R}}( |\nabla u|^{2}+|\nabla v|^{2})\\
&\quad
+\frac{R}{2} \int_{\partial B_{R}}(u^{2}+v^{2})
+\frac{\lambda R}{2}\int_{\partial B_{R}}\phi (u^{2}+v^{2})
- \frac{R}{2q}\int_{\partial B_{R}} (|u|^{2q}+|v|^{2q})
+ \frac{\lambda}{8\pi R} \int_{\partial B_{R}}|x\cdot\nabla \phi|^{2}\\
&\quad
-\frac{R\lambda}{16 \pi} \int_{\partial B_{R}}|\nabla \phi|^{2}
-\frac{\beta R}{q}\int_{\partial B_{R}}|u|^{q}|v|^{q}.
\end{align*}
Arguing as in \cite[pag. 321]{BL}, there exists  a suitable sequence $R_{n}\to+\infty$ on which the right hand side 
above tends to zero. 
Thus, passing to the limit we deduce that
\begin{align*}
&\frac12 (\|\nabla u\|_2^{2} + \|\nabla v\|_2^{2})
+ \frac32 (\|u\|_2^{2} +  \|v\|_2^{2})
-\frac{\lambda}{16\pi}\|\nabla \phi\|_2^{2}
+\frac{3}{2}\lambda \int (u^{2}+v^{2})\phi\\
&\qquad\qquad
=\frac{3}{2q} \Big(\|u\|_{2q}^{2q}+\|v\|_{2q}^{2q} +2\beta \int |u|^{q}|v|^{q}\Big).
\end{align*}
Hence, using \eqref{lastint}, we achieve the conclusion.
\end{proof}

With the Pohozaev identity \eqref{Pohosyst}, we can show easily our nonexistence result. Indeed we have

% ENUNCIATO TH. NONEXISTENCE
%\begin{Th}
%System \eqref{system} has only the trivial solution in ${\rm H} \cap (L^{2q}(\mathbb{R}^3)\times L^{2q}(\mathbb{R}%^3))$ for $q\in [1/2,1]\cup [3,+\infty[ $.
%\end{Th}
%\begin{proof}
\begin{proof}[Proof of Theorem \ref{th:nonexistence}]
Let $(u,v)\in {\rm H} \cap (L^{2q}(\mathbb{R}^3)\times L^{2q}(\mathbb{R}^3))$ be a nontrivial solution of \eqref{system} for $q\in [1/2,1]\cup [3,+\infty[ $.
Using the Nehari identities \eqref{N1} and \eqref{N2} and the Pohzaev identity \eqref{Pohosyst} we have 
\begin{align*}
0
&=
\|\nabla u \|_2^2 + \|\nabla v \|_2^2
+ \| u \|_2^2 + \| v \|_2^2 
+ \lambda  \int  (u^2 + v^2) \phi_{u,v}
- \|u\|_{2q}^{2q} - \|v\|_{2q}^{2q}
-2\beta\int |u|^q |v|^q\\
%(Pohozaev)
%&{\color{red}=
%	\|\nabla u \|_2^2 + \|\nabla v \|_2^2
%	+ \| u \|_2^2 + \| v \|_2^2 
%	+ \lambda  \int \phi_{u,v}(u^2 + v^2)}\\
%&\quad
%{\color{red}
%	- \frac{2q}{3}
%	\left\{
%	\frac{1}{2} (\|\nabla u \|_2^2 + \|\nabla v \|_2^2)
%	+ \frac{3}{2} (\| u \|_2^2 + \| v \|_2^2)
%	+ \frac{5}{4}\lambda  \int \phi_{u,v}(u^2+v^2)
%	\right\}
%}\\
&=
\left(1- \frac{q}{3}\right) (\|\nabla u \|_2^2 + \|\nabla v \|_2^2)
+ (1-q) (\| u \|_2^2 + \| v \|_2^2)
+ \left(1- \frac{5}{6}q\right)\lambda  \int  (u^2+v^2) \phi_{u,v},
\end{align*}
which is strictly negative for $q\geq 3$ or strictly positive for $q\leq 1$ and so we reach a contradiction.
\end{proof}

\subsection{Existence of a radial ground state}\label{32}

Here %\textcolor{red}{ following the arguments in \cite{RuizJFA}, }
	we  find a radial ground state solution for our system \eqref{system}. As we have stated in the Introduction, to get compactness we restrict ourselves to radial functions. Thus, from now on, we will consider $\textrm H_{\textrm{r}}$ as functional space, even if several facts do not require symmetry assumptions.

We start showing that, as in \cite[Lemma 2.1]{RuizJFA}, the following properties hold.
%\todo[inline]{Qui dobbiamo dire chi \`e $q$.}
\begin{lemma}\label{lem:Phi}
Let $q\in(1,3)$ and $\{(u_{n}, v_{n})\}\subset {\rm H}_{\emph r}$ be such that $(u_{n}, v_{n})\rightharpoonup (u,v)$ in ${\rm H}_{\emph r}$ as $n\to +\infty$. We have, as $n\to +\infty$,
\begin{align}
\phi_{u_{n}, v_{n}}
&\to
\phi_{u,v} \text{ in }D_{\emph r}^{1,2}(\mathbb R^{3}),\label{eq:10} \\
\int (u_{n}^{2}+v_{n}^{2})\phi_{u_{n}, v_{n}}
&\to
\int (u^{2}+v^{2}) \phi_{u, v} ,\label{eq:20}\\
\int |u_{n}|^{q}|v_{n}|^{q}
&\to
\int |u|^{q} |v|^{q}.\label{eq:30}
\end{align}
\end{lemma}

\begin{proof}
Let us define on $D^{1,2}_{\textrm r}(\mathbb R^{3})$ the linear and continuous operators
\begin{align*}
T_{n}(w)
&:= \int  \nabla w \nabla \phi_{u_n,v_n}
\left(= 4\pi \int (u_{n}^{2}+v_{n}^{2})w\right),\\
T(w)
&:= \int  \nabla w \nabla \phi_{u,v}
\left( =  4\pi \int (u^{2} + v^{2})w\right).
\end{align*}
%They are also continuous since they are defined through the scalar product in $D^{1,2}(\mathbb R^{3})$.\\
Then, due to the compact embedding of the radial functions we have
$$|T_{n}(w) - T(w)|
\leq 4\pi \|w\|_{6} \Big(\|u_{n}^{2} - u^{2}\|_{6/5}+ \|v_{n}^{2} - v^{2}\|_{6/5}\Big)
\leq  \varepsilon_{n} \|\nabla w\|_2. $$
Hence 
 $T_{n}- T \to 0$ as  operators on $D_{\textrm{r}}^{1,2}(\mathbb R^{3})$, and  by  the Riesz Theorem
 this implies \eqref{eq:10}.\\
%Dalla disuguaglianza di sopra, passando al sup, otteniamo
%\[
%\|T_n -T\|\to 0.
%\]
%Inoltre, dal Teorema di Riesz abbiamo
%\[
%\|T_n - T \|=\|\nabla (\phi_n - \phi)\|_2
%\]
%e concludiamo.
Convergence \eqref{eq:20} follows from
$$\phi_{u_{n}, v_{n}} \to \phi_{u,v} \text{ in } L^{6}(\mathbb R^{3}) \quad \text{and}\quad
u_{n}^{2}+v_{n}^{2} \to u^{2}+v^{2}\text{ in } L^{6/5}(\mathbb R^{3}).$$
Finally, to get \eqref{eq:30}, we observe that, using again the compact embedding of the radial functions,
%(a-b)^2\leq(c+b)^2 with 0<a\leq c
$$\| |u_{n}|^{q} - |u|^{q}\|_{2}, \ \  \| |v_{n}|^{q} - |v|^{q}\|_{2}\to 0.$$
Thus
\begin{align*}
\Big|\int |u_{n}|^{q} |v_{n}|^{q} - |u|^{q}|v|^{q} \Big|
&\leq
\int |u_{n}|^{q}  \Big | |v_{n}|^{q} - |v|^{q} \Big | + \int |v|^{q}  \Big| |u_{n}|^{q} - |u|^{q}  \Big|\\
&\leq
\|u_{n}\|_{2q}^{q}  \| | v_{n}|^{q} - |v|^{q} \|_{2}  + \|v\|_{2q}^{q}  \| |u_{n}|^{q} - |u|^{q} \|_{2}=\varepsilon_{n},
\end{align*}
concluding the proof.
\end{proof}

%\todo[inline]{Forse gi\`a detto.}
%{\color{red}In view of our scopes, from now on we assume $q\in (3/2,3)$ and, even if some properties hold true %without assuming any symmetry on the functions, since we are interested into the radial case, which is a natural %constraint due to the Palais Principle of Symmetric Criticality \cite{Palais}, we work directly on radial functions.}

Let us consider now the {\em Nehari-Pohozaev  manifold}
\[
\mathcal{M}:=\left\{(u,v)\in \textrm{H}_{\textrm r} \setminus \{0\} : J_{\lambda,\beta}(u,v)=0 \right\}
\]
where
\begin{equation*}\label{eq:J}
\begin{split}
J_{\lambda,\beta}(u,v)
&:=
\frac{3}{2} (\|\nabla u \|_2^2 + \|\nabla v \|_2^2)
+ \frac{1}{2} (\| u \|_2^2 + \| v \|_2^2)
+ \frac{3}{4}\lambda  \int (u^2+v^2) \phi_{u,v} \\
&\quad
- \frac{4q-3}{2q} \left(
\|u\|_{2q}^{2q}+\|v\|_{2q}^{2q}
+2\beta \int |u|^q |v|^q
\right).
\end{split}
\end{equation*}

Observe that the condition $J_{\lambda,\beta}(u,v)=0$ can be obtained by a {\em linear combination} of the Nehari \eqref{N1}, \eqref{N2} and Pohozaev \eqref{Pohosyst} identities. Thus, $\mathcal{M}$ contains all  nontrivial radial critical points of $I_{\lambda,\beta}$.\\
%{\color{red}\begin{Cor}
%$\mathcal{M}$ contains all  nontrivial critical points of $I_{\lambda,\beta}$.
%\end{Cor}
%\begin{proof}
%This follows by the fact that any critical point $(u,v)$ of $I_{\lambda,\beta}$ satisfies the 
%{\sl Nehari identities}
%\begin{eqnarray*}%\label{eq:Ne}
%&&\|\nabla u\|_{2}^{2} + \|u\|_{2}^{2} + \int_{\mathbb R^{3}} \phi_{u,v} u^{2} =
%\|u\|_{2q}^{2q}+\beta \int_{\mathbb R^{3}}|u|^q |v|^{q} \\
%&&\|\nabla v\|_{2}^{2} + \|v\|_{2}^{2} + \int_{\mathbb R^{3}} \phi_{u,v} v^{2} =
%\|v\|_{2q}^{2q}+\beta \int_{\mathbb R^{3}}|u|^q |v|^{q} 
%\end{eqnarray*}
%and the Pohozaev identity given in Lemma \ref{lem:Pohozaev}
%and $J_{\lambda,\beta} =2N -P$.
%\end{proof}}
%
Moreover, the following simple result assures us that any  couple $(u,v)\in {\rm H}_{\textrm r} \setminus\{0\}$ can be uniquely projected on $\mathcal M$ via $\gamma_{u,v}$ (see its definition in \eqref{gamma}) and gives us a further property of such a projection.
\begin{Lem}\label{Step0}
For any $(u,v)\in {\rm H}_{\emph r} \setminus\{0\}$
there exists a unique $t_{u,v}>0$
such that $\gamma_{u,v}(t_{u,v})\in \mathcal M$ and
\begin{equation}
\label{formulaStep0}
I_{\lambda,\beta} (\gamma_{u,v}(t_{u,v}))= \max_{t>0} I_{\lambda,\beta} (\gamma_{u,v}(t)).
\end{equation}
\end{Lem}
\begin{proof}
The existence and uniqueness of $t_{u,v}$ is an easy consequence of (\ref{fa}) in Lemma \ref{lem:calculus}, since
\begin{equation*}
\begin{split}
J_{\lambda,\beta}(\gamma_{u,v}(t))
&=
\frac{3}{2} t^3 (\|\nabla u \|_2^2 + \|\nabla v \|_2^2)
+ \frac{t}{2} (\| u \|_2^2 + \| v \|_2^2)
+ \frac{3}{4} \lambda t^3 \int (u^2+v^2) \phi_{u,v}\\
&\quad
- \frac{4q-3}{2q}t^{4q-3}\left(
\|u\|_{2q}^{2q}+\|v\|_{2q}^{2q}
+2\beta\int |u|^q |v|^q
\right)
\end{split}
\end{equation*}
and $q>3/2$.\\
Moreover, since
\[
J_{\lambda,\beta}(\gamma_{u,v}(t)) = t \frac{d}{dt} I_{\lambda,\beta} (\gamma_{u,v}(t)),
\]
we have that $t_{u,v}$ is the unique strictly positive critical point of $I_{\lambda,\beta} (\gamma_{u,v}(t))$ and so, again by (\ref{fa}) in Lemma \ref{lem:calculus}, we conclude.
\end{proof}

%\textcolor{blue}{
%\todo[inline]{Invece di cominciare la frase con Then, io farei riferimento all'ultima formula.}
%Then we can write also
%$$
%\mathcal M = \left\{ (u,v) \in \textrm H\setminus\{0\}:\frac{d}{dt} I_{\lambda,\beta} (\gamma_{u,v}(t))\Big|_{t=1}
%=0 \right\}.$$
%}

Now we are ready to find the ground state solutions of \eqref{system} by minimizing the functional $I_{\lambda,\beta}$ on $\mathcal M$. 

\begin{proof}[Proof of Theorem \ref{th:gs} (existence of a ground state)]
We divide the proof in several steps.\\
{\bf Step 1:} {\it $\mathcal M$ is bounded away from zero, i.e. $(0,0)\notin \partial\mathcal M$.}\\
Let $(u,v)\in \mathcal M$. Since 
$$2\int|u|^{q}|v|^{q}\leq \|u\|_{2q}^{2q} +\|v\|_{2q}^{2q}\leq C\Big( \|u\|^{2} + \|v\|^{2}\Big)^{q}$$
we deduce
%\todo[inline]{Usiamo anche la positivit\`a del termine nonlocale.}
\[
\frac12(\|u\|^{2} +\|v\|^{2})\leq  C \Big( \|u\|^{2} + \|v\|^{2} \Big)^{q},
\]
so that there exists $\rho>0$ such that $\|u\|^{2}+\|v\|^{2}\geq \rho>0$
%{\color{red}(for a suitable $\rho$ depending on $\beta$ but independent on $u,v$)}
and the conclusion holds.\\
{\bf Step 2:} $\mathfrak m_{\beta}:=\inf_{\mathcal M} I_{\lambda,\beta}>0$.\\
For $(u,v)\in \mathcal M$ we set, for simplicity,
\begin{equation*}
\begin{cases}\label{eq:abcd0}
a:=\|\nabla u\|_{2}^{2} + \|\nabla v\|_{2}^{2}, & b:=\|u\|_{2}^{2}+\|v\|_{2}^{2}, \vspace{5pt}\\
c:= \lambda\displaystyle\int (u^{2}+v^{2}) \phi_{u,v}, & \displaystyle d:=\|u\|_{2q}^{2q}+\|v\|_{2q}^{2q}+2\beta\int|u|^{q}|v|^{q}.
\end{cases}
\end{equation*}
If $k:=I_{\lambda,\beta}(u,v)$,  we have
\begin{equation*}
\begin{cases}
\displaystyle\frac{1}{2}a+\frac{1}{2}b+\frac{1}{4}c - \frac{1}{2q}d = k\medskip \\
\displaystyle\frac{3}{2}a+\frac{1}{2}b+\frac{3}{4}c - \frac{4q-3}{2q}d=0.
\end{cases}
\end{equation*}
In terms of $a,b,k$ the unknown $c$ is given by
$$0<c=2 \frac{(4q-3)k-(2q-3)a-2b(q-1)}{2q-3}.$$
Then taking into account Step 1, 
we have
\begin{equation}\label{eq:limitazione}
(2q-3)\rho<(2q-3) (a+b)<(2q-3)a+2(q-1)b<(4q-3)k,
\end{equation}
where $\rho>0$ is the constant found at the end of the previous step, meaning that $k$ is bounded away from zero.\\
% above a constant which does not depend on $(u,v)$.
{\bf Step 3:} {\it If $\{(u_{n},v_{n})\}$ is a  minimizing sequence for $I_{\lambda,\beta}$ on $\mathcal M$, then it is bounded. Hence, up to subsequence, it weakly converges to some $(\mathfrak  u_{\beta}, \mathfrak v_{\beta})$ in ${\rm H}_{\emph r}$.}\\
Let $\{(u_{n},v_{n})\}\subset \mathcal M$ such that $k_{n}:=I_{\lambda,\beta}(u_{n}, v_{n})\to \mathfrak m_{\beta}$. 
Setting for simplicity
\begin{equation}
\begin{cases}\label{eq:abcd}
\displaystyle a_{n}:=\|\nabla u_{n}\|_{2}^{2} + \|\nabla v_{n}\|_{2}^{2},
&   b_{n}:=\|u_{n}\|_{2}^{2}+\|v_{n}\|_{2}^{2}, \vspace{5pt}  \\
\displaystyle c_{n}:= \lambda\int (u_{n}^{2}+v_{n}^{2})\phi_{u_{n},v_{n}},
& \displaystyle d_{n}:=\|u_{n}\|_{2q}^{2q}+\|v_{n}\|_{2q}^{2q}+2\beta\int|u_{n}|^{q}|v_{n}|^{q},
\end{cases}
\end{equation}
arguing as in Step 2, see \eqref{eq:limitazione}, we get
$$(2q-3)(a_{n}+b_{n})<(4q-3)k_{n}\to (4q-3)\mathfrak m_{\beta}$$
and so the minimising sequence $\{(u_{n},v_{n})\}$ is bounded.\\
{\bf Step 4:}  {\it $\{(u_{n}, v_{n})\}$ strongly converges to $(\mathfrak u_{\beta}, \mathfrak v_{\beta})$ in ${\rm H}_{\emph r}$. Then 
$(\mathfrak u_{\beta}, \mathfrak v_{\beta})\in \mathcal M$ and it minimizes $I_{\lambda,\beta}$.}\\
Here is the scenario in which we need the radial setting.\\
Observe that, by the previous step, it follows that
\begin{equation}\label{eq:convL2}
	u_{n}\rightharpoonup \mathfrak u_{\beta}, \ \ v_{n}\rightharpoonup \mathfrak v_{\beta}, \text{ in } L^{2}(\mathbb R^{3}) \text{ and  in } D^{1,2}(\mathbb{R}^3)
	\end{equation}
and, eventually passing to a suitable subsequence,
	\begin{equation}\label{wls}
	\|\nabla \mathfrak u_{\beta}\|_{2}^{2}\leq \lim_{n}\|\nabla u_{n}\|^{2}_{2}, \quad  \|\nabla \mathfrak v_{\beta}\|_{2}^{2}\leq \lim_{n}\|\nabla v_{n}\|_{2}^{2}, \quad 
	\| \mathfrak u_{\beta}\|_{2}^{2}\leq \lim_{n}\| u_{n}\|^{2}_{2}, \quad \| \mathfrak v_{\beta}\|_{2}^{2}\leq \lim_{n}\| v_{n}\|_{2}^{2}.
	\end{equation}
Maintaining the  notations in \eqref{eq:abcd}, we define
$$
\overline a:=\lim_{n} a_{n}\,, \ \ \overline b:=\lim_{n} b_{n}\,, \ \ 
\overline c:=\lim_{n} c_{n} \,, \ \ \overline d:= \lim_{n} d_{n},
$$
where we are  assuming that the  limits exists (eventually passing to suitable subsequences) being $\{a_{n}\}$, $\{b_{n}\}$, $\{c_{n}\}$, $\{d_{n}\}$ bounded sequences (see the previous Step).\\
Observe also that, by Step 1,
\begin{equation}
\label{sumbar>0}
\overline{a}+\overline{b}>0.
\end{equation}
Moreover, the relations 
	$$I_{\lambda,\beta}(u_{n}, v_{n})\to  \mathfrak m_{\beta}\quad \text{and} \quad  J_{\lambda,\beta}(u_{n},v_{n})=0$$
	give
	\begin{equation}\label{eq:barre}
	\begin{cases}
	\displaystyle\frac{1}{2}\overline{a}+\frac{1}{2}\overline{b}+\frac{1}{4}\overline{c} - \frac{1}{2q}\overline{d} =  \mathfrak m_{\beta} \medskip \\
	\displaystyle\frac{3}{2}\overline{a}+\frac{1}{2}\overline{b}+\frac{3}{4}\overline{c} - \frac{4q-3}{2q}\overline{d}=0.
	\end{cases}
	\end{equation}
Thus, by the second equation in \eqref{eq:barre} and \eqref{sumbar>0} we get $\overline d>0$.\\
Hence, using an analogous notation as before for the pair $(\mathfrak{u}_{\beta},\mathfrak{v}_{\beta})$, namely
\begin{equation}
\begin{cases}\label{eq:abcdGS}
a:=\|\nabla \mathfrak u_{\beta}\|_{2}^{2} + \|\nabla \mathfrak v_{\beta}\|_{2}^{2}, &  b:=\|\mathfrak u_{\beta}\|_{2}^{2}+\|\mathfrak v_{\beta}\|_{2}^{2},\vspace{5pt}\\
c:= \lambda\displaystyle\int (\mathfrak u_{\beta}^{2}+\mathfrak v_{\beta}^{2}) \phi_{\mathfrak u_{\beta}, \mathfrak v_{\beta}} , & \displaystyle d:=\|\mathfrak u_{\beta}\|_{2q}^{2q}+\|\mathfrak v_{\beta}\|_{2q}^{2q}+2\beta\int|\mathfrak u_{\beta}|^{q}|\mathfrak v_{\beta}|^{q},
\end{cases}
\end{equation}
by \eqref{wls}, we have
\begin{equation}\label{eq:aabb}
a\leq \overline a \quad \text{and} \quad b\leq\overline b.
\end{equation}
Observe  that, due to Lemma \ref{lem:Phi} and to the compact embedding in the radial setting 
$$ \overline c=c \quad \text{and} \quad \overline d = d.$$
If $a+b<\overline a+\overline b$,
then, taking into account that $J_{\lambda,\beta}(u_{n},v_{n})=0$, we have that 
$J_{\lambda,\beta}(\mathfrak u_{\beta}, \mathfrak  v_{\beta})
%< {\color{red}J_{\lambda,\beta}(u_{n}, v_{n}) =}
<0$, meaning that $(\mathfrak u_{\beta}, \mathfrak v_{\beta})\notin\mathcal M$ and that $(\mathfrak u_{\beta}, \mathfrak v_{\beta})\neq (0,0)$.
This implies that $a,b,c,d>0$ and, by \eqref{eq:aabb}, also $\overline a, \overline b>0$.\\
Moreover, by Lemma \ref{Step0} there exists a unique $t_{\mathfrak u_{\beta}, \mathfrak v_{\beta}}>0$ such that $\gamma_{\mathfrak u_{\beta}, \mathfrak v_{\beta}}(t_{\mathfrak u_{\beta}, \mathfrak v_{\beta}})\in \mathcal M$ (see \eqref{gamma}).\\
Consider now, for $t\geq0$, the functions
\[
f(t)
= \frac{t^3}{2} {a}
+ \frac{t}{2} {b}
+ \frac{t^3}{4} {c}
- \frac{t^{4q-3}}{2q} {d},
\qquad
\overline{f}(t)
= \frac{t^3}{2} \overline{a}
+ \frac{t}{2} \overline{b}
+ \frac{t^3}{4} \overline{c}
- \frac{t^{4q-3}}{2q} \overline{d}.
\]
Note that
\[
f(t)= I_{\lambda,\beta} (\gamma_{\mathfrak u_{\beta}, \mathfrak v_{\beta}}(t))
\quad\text{and}\quad
tf'(t)=J_{\lambda,\beta} (\gamma_{\mathfrak u_{\beta}, \mathfrak v_{\beta}}(t)).
\]
The functions $f$ and $\overline{f}$ have both a unique critical point corresponding to the global maximum (see (\ref{fa}) in Lemma \ref{lem:calculus}).
In particular, the global maximizer of $f$ is $t_{\mathfrak u_{\beta}, \mathfrak v_{\beta}}$ and, by \eqref{eq:barre}, we deduce that $\overline f$ achieves the maximum in $t=1$. Moreover, since we are assuming
$a+b< \overline a+ \overline b$, it holds $f(t)<\overline f(t)$ for $t>0$. Hence $\gamma_{\mathfrak u_{\beta}, \mathfrak v_{\beta}}(t_{\mathfrak u_{\beta}, \mathfrak v_{\beta}}) \in \mathcal M$ and
\[
I_{\lambda,\beta}(\gamma_{\mathfrak u_{\beta}, \mathfrak v_{\beta}}(t_{\mathfrak u_{\beta}, \mathfrak v_{\beta}}))
=f(t_{\mathfrak u_{\beta}, \mathfrak v_{\beta}})
< \max_{t\geq 0} \overline f(t) = \mathfrak m_{\beta},
\]
which is a contradiction.\\
Hence, by \eqref{eq:aabb},  we infer $a=\overline a$ and $b=\overline b$,
so that, using \eqref{eq:convL2}, we get
$(u_{n},v_{n})\to (\mathfrak u_{\beta}, \mathfrak v_{\beta})$ in $\textrm{H}_{\textrm r}$.\\
{\bf Step 5:} {\em $(\mathfrak u_{\beta}, \mathfrak v_{\beta})$ is a regular point of $\mathcal M$, i.e. $J_{\lambda,\beta}'(\mathfrak u_{\beta}, \mathfrak v_{\beta})\neq 0$.}\\
Assume  by contradiction that $J_{\lambda,\beta}'(\mathfrak u_{\beta}, \mathfrak v_{\beta})=0$ so that we have
\begin{equation}\label{eq:barre2}
\begin{cases}
-3\Delta \mathfrak u_{\beta}+\mathfrak u_{\beta}+3\lambda \phi_{\mathfrak u_{\beta}, \mathfrak v_{\beta}} \mathfrak u_{\beta} - (4q-3)\left( |\mathfrak u_{\beta}|^{2q-2}
+\beta|\mathfrak u_{\beta}|^{q-2}|\mathfrak v_{\beta}|^{q} \right)\mathfrak u_{\beta}=0\medskip \\
-3\Delta \mathfrak v_{\beta}+\mathfrak v_{\beta}+3\lambda \phi_{\mathfrak u_{\beta},\mathfrak v_{\beta}} \mathfrak v_{\beta} - (4q-3)\left( |\mathfrak v_{\beta}|^{2q-2}+\beta|\mathfrak v_{\beta}|^{q-2}|\mathfrak u_{\beta}|^{q}
\right)\mathfrak v_{\beta} =0.
\end{cases}
\end{equation}
Then, under the notations \eqref{eq:abcdGS}, we have
\begin{equation*}
     \left\{
        \begin{array}{ll}
      \displaystyle{\frac{1}{2}a} +\displaystyle{\frac{1}{2}b}+\displaystyle{\frac{1}{4}c} -\frac{1}{2q}d =\mathfrak  m_{\beta}, \medskip \\
\displaystyle\frac{3}{2}a+\displaystyle\frac{1}{2}b+\displaystyle\frac{3}{4}c-\displaystyle\frac{4q-3}{2q} d=0, \medskip \\
3a+b+3c-(4q-3)d=0, \medskip \\
\displaystyle\frac{3}{2}a+\displaystyle\frac{3}{2}b+\displaystyle\frac{15}{4}c - \displaystyle3\frac{4q-3}{2q}d=0,
        \end{array}
        \right.
\end{equation*}        
where the third equation is simply $J_{\lambda,\beta}'(\mathfrak u_{\beta},\mathfrak v_{\beta})[\mathfrak u_{\beta},\mathfrak v_{\beta}]=0$ and, finally, the fourth equation is the Pohozaev identity  for \eqref{eq:barre2}.
%\begin{equation}\label{eq:}
%\frac{d}{dt}F(\gamma_{\mathfrak u,\mathfrak v}(t))\Big|_{t=1}=0
%\end{equation}
%where $F$ is the functional
%\begin{multline}
%F(u,v)=\frac{3}{2}\left( \|\nabla  u\|_{2}^{2} + \|\nabla  v\|^{2}_{2}\right)
%+\frac{1}{2}\left(\| u\|_{2}^{2} +\| v\|_{2}^{2}\right)
%+\frac{3}{4}\lambda\int \phi_{ u,  v}\\
%-\frac{4q-3}{2q}\left( \| u\|_{2q}^{2q} + \| v\|_{2q}^{2q}
% + 2\beta \int | u|^{q}| v|^{q}\right).
%\end{multline}
The solution of the above system is given by
\[
a= -\frac{4q-3}{4(2q-3)}\mathfrak m_{\beta},
\quad
b= 3\frac{4q-3}{4(q-1)}\mathfrak m_{\beta},
\quad
c = -\frac{4q-3}{2(2q-3)}\mathfrak m_{\beta},
\quad
d= - \frac{3q}{4(2q-3)(q-1)}\mathfrak m_{\beta}.
\]
Since $q\in(3/2, 3)$, then $a<0$, which is impossible.\\
{\bf Step 6:} $I_{\lambda,\beta}'(\mathfrak u_{\beta},\mathfrak v_{\beta})=0$.\\
Thanks to the Lagrange multiplier rule we know that, for some $\ell\in \mathbb R$,
$$I_{\lambda,\beta}'(\mathfrak u_{\beta},\mathfrak v_{\beta})
= \ell J_{\lambda,\beta}'(\mathfrak u_{\beta},\mathfrak v_{\beta}).$$
We want to show that $\ell=0$.\\
By expliciting  the above identity we get
%\begin{multline*}%\label{eq:}
%-\Delta \mathfrak u+\mathfrak u +\lambda \phi_{\mathfrak u,\mathfrak v} \mathfrak u-|\mathfrak u|^{2q-2}\mathfrak u-\beta|\mathfrak u|^{q-2}\mathfrak u|\mathfrak v|^{q}= \\
%\ell \left( -3\Delta \mathfrak u+\mathfrak u +3\lambda \phi_{\mathfrak u,\mathfrak v} \mathfrak u
%-(4q-3)|\mathfrak u|^{2q-2}\mathfrak u-\beta(4q-3)|\mathfrak u|^{q-2}\mathfrak u|\mathfrak v|^{q}\right)
%\end{multline*}
%and
%\begin{multline*}%\label{eq:}
%-\Delta  \mathfrak v+\mathfrak v +\lambda \phi_{\mathfrak u,\mathfrak v}\mathfrak  v
%-|\mathfrak v|^{2q-2}\mathfrak v-\beta|\mathfrak v|^{q-2}\mathfrak v|\mathfrak u|^{q}= \\
%\ell \left( -3\Delta \mathfrak v+\mathfrak v +3\lambda \phi_{\mathfrak u,\mathfrak v}\mathfrak v
%-(4q-3)|\mathfrak v|^{2q-2}\mathfrak v-\beta(4q-3)|\mathfrak v|^{q-2}\mathfrak v|\mathfrak u|^{q}\right).
%\end{multline*}
%These can be rewritten as
\begin{equation}\label{eq:1}
	-(3\ell-1)\Delta \mathfrak u_{\beta}
	+(\ell-1)\mathfrak u_{\beta}
	+(3\ell-1)\lambda\phi_{\mathfrak u_{\beta},\mathfrak v_{\beta}}\mathfrak u_{\beta}
	-((4q-3)\ell-1)
	\big[ |\mathfrak u_{\beta}|^{2q-2}
	+\beta|\mathfrak u_{\beta}|^{q-2}|\mathfrak v_{\beta}|^{q}\big]
	\mathfrak u_{\beta}=0
	\end{equation}
	and
	\begin{equation}\label{eq:2}
	-(3\ell-1)\Delta \mathfrak v_{\beta}
	+(\ell-1)\mathfrak v_{\beta}
	+(3\ell-1)\lambda\phi_{\mathfrak u_{\beta},\mathfrak v_{\beta}}\mathfrak v_{\beta}
	- ((4q-3)\ell-1)
	\big[ |\mathfrak v_{\beta}|^{2q-2}
	+\beta|\mathfrak v_{\beta}|^{q-2} |\mathfrak u_{\beta}|^{q}\big]
	\mathfrak v_{\beta}=0.
	\end{equation}
Now, multiplying \eqref{eq:1} and \eqref{eq:2} by $\mathfrak u_{\beta}$ and  $\mathfrak v_{\beta}$ respectively, integrating, and, finally, summing up the two identities obtained, we get, using again the notations in \eqref{eq:abcdGS},
\begin{equation*}\label{eq:lagrange}
(3\ell-1)a +(\ell -1)b+(3\ell-1)c-((4q-3)\ell-1)d=0.
\end{equation*} 
On the other hand, arguing as in Lemma \ref{lem:Pohozaev}, we can associate to \eqref{eq:1} and \eqref{eq:2} the  Pohozaev identity
\[
\frac{3\ell-1}{2}a+\frac{3}{2}(\ell-1)b+\frac54(3\ell-1)c-\frac{3}{2q}((4q-3)\ell-1)d=0.
\]
Then $a,b,c,d$ satisfy the system
\begin{equation*}
 \left\{
        \begin{array}{ll}
      \displaystyle{\frac{1}{2}a} +\displaystyle{\frac{1}{2}b}+\displaystyle{\frac{1}{4}c} -\frac{1}{2q}d = \mathfrak m_{\beta}, \medskip \\
\displaystyle\frac{3}{2}a+\displaystyle\frac{1}{2}b+\displaystyle\frac{3}{4}c-\displaystyle\frac{4q-3}{2q} d=0, \medskip \\
(3\ell-1)a +(\ell -1)b+(3\ell-1)c-((4q-3)\ell-1)d=0, \medskip \\
\displaystyle\frac{3\ell-1}{2}a+\frac{3}{2}(\ell-1)b+\frac54(3\ell-1)c-\frac{3}{2q}((4q-3)\ell-1)d=0.
        \end{array}
        \right.
\end{equation*}
%\todo[inline]{Codice per Wolfram: ${(1/2, 1/2, 1/4, -1/(2q)), (3/2, 1/2, 3/4, -(4q-3)/(2q)), (3l-1, l -1, 3l-1, -((4q-3)l-1), ((3l-1)/2, 3/2(l-1), 5/4(3l-1), -3/(2q)((4q-3)l-1))} $.}
The determinant of the matrix of the coefficients is
\[
-\frac{\ell (3\ell-1)(q-1)(2q-3)}{q}.
\]
The assumptions on $q$ imply that $q-1\neq 0$ and $2q-3 \neq 0$. Moreover, also $\ell \neq 1/3$. Indeed, if it were  
$\ell=1/3$, the third equation of the system above would be
$$-\frac23b - \frac{2(2q-3)}{3}d=0$$
which is impossible since $b,d>0$. Thus, if it were also $\ell\neq0$, then the determinant would be different from zero, meaning that the system would have a unique solution. In particular
$$d= -\frac{3q}{4(q-1)(2q-3)}\mathfrak m_{\beta}<0$$
which is impossible. Summing up it yields $\ell=0$, concluding the proof of the Step.
\end{proof}

\begin{Rem}\label{rem:positivita}
 Without loss of generality, since $(|\mathfrak u_{\beta}|,|\mathfrak v_{\beta}|)$ 
 is also a solution at the level $\mathfrak m_{\beta}$, applying the Maximum Principle,
  we can assume that, whenever
$ \mathfrak u_{\beta},\mathfrak v_{\beta}$  are nontrivial, they are strictly positive.
 \end{Rem}

\begin{Rem}\label{rem:Convergenza}
For future reference, we observe  that the statements of previous Steps 5 and 6 and the inequality
$$0<\|\mathfrak u_{\beta}\|^2 + \|\mathfrak v_{\beta}\|^{2} \leq
\frac{ 4q-3}{2q-3}\mathfrak m_{\beta},% I_{\lambda,\beta}(\mathfrak u_{\beta},\mathfrak v_{\beta}).
$$
which follows by \eqref{eq:limitazione} in Step 2, hold for any pair of ground states.
\end{Rem}

Moreover, as an immediate consequence of Theorem \ref{th:gs} what we have just seen 
 we can prove the following further result.
\begin{Cor}
Let $(\mathfrak u_{\beta}, \mathfrak v_{\beta})\in {\rm H}_{\emph r}\setminus\{0\}$ be a ground state found in Theorem \ref{th:gs}. We have that
\[
%I_{\lambda,\beta}(\mathfrak u_{\beta}, \mathfrak v_{\beta}) 
\mathfrak m_{\beta}
=I_{\lambda,\beta}(\gamma_{\mathfrak u_{\beta}, \mathfrak v_{\beta}}(1))
= \inf_{(u,v)\in {\rm H}_{\emph r}\setminus\{0\} } \max_{t>0} I_{\lambda,\beta}(\gamma_{u,v}(t))
\]
\end{Cor}
%\todo[inline]{Caratterizzazione con  l'inf sulla classe di curve (forse non serve, vedere alla fine).}
\begin{proof}
By Lemma \ref{Step0} we have that, for every $(u,v)\in {\rm H}_{\textrm r}\setminus\{0\}$,
\[
\mathfrak m_{\beta}  %I_{\lambda,\beta}(\mathfrak u_{\beta}, \mathfrak v_{\beta})
= \min_{(u,v)\in \mathcal M} I_{\lambda,\beta}(u,v)
\leq I_{\lambda,\beta} (\gamma_{u,v}(t_{u,v}))
=\max_{t>0} I_{\lambda,\beta} (\gamma_{u,v}(t)).
\]
Then, passing to the infimum on $(u,v)\in {\rm H}_{\textrm r}\setminus\{0\}$, we get
\[
\mathfrak m_{\beta}
%I_{\lambda,\beta}(\mathfrak u_{\beta}, \mathfrak v_{\beta})
\leq \inf_{(u,v)\in {\rm H}_{\textrm r}\setminus\{0\}} \max_{t>0} I_{\lambda,\beta}(\gamma_{u,v}(t))
\leq \max_{t>0} I_{\lambda,\beta}(\gamma_{\mathfrak u_{\beta}, \mathfrak v_{\beta}}(t))
= I_{\lambda,\beta}(\gamma_{\mathfrak u_{\beta}, \mathfrak v_{\beta}}(1))
= 
% I_{\lambda,\beta}(\mathfrak u_{\beta}, \mathfrak v_{\beta})
\mathfrak m_{\beta}
\]
concluding the proof.
\end{proof}
%
%\vspace{2cm}
%
%\begin{eqnarray*}%\label{eq:}
%-(1-3\mu)\Delta u+(1-\mu) u+\lambda(1-3\mu)\phi_{u,v}u-(1-\mu(4q-3))|u|^{2q-2}u
%-\beta\left( 1-\mu(4q-3)\right)|u|^{q-2}u|v|^{q}=0\\
%-(1-3\mu)\Delta v+(1-\mu) v+\lambda(1-3\mu)\phi_{u,v}v-(1-\mu(4q-3))|v|^{2q-2}v
%-\beta\left( 1-\mu(4q-3)\right)|v|^{q-2}v|u|^{q}=0
%\end{eqnarray*}
%
%
%\vspace{3cm}
%
%\begin{eqnarray*}
%\frac{3}{2}|\nabla u|_{2}^{2}+\frac{1}{2}|u|_{2}^{2} + \frac{3}{2}|\nabla v|_{2}^{2}+\frac{1}{2}|v|_{2}^{2} 
%+\frac{3\lambda}{4}\int \Phi(u,v)(u^{2}+v^{2})-\frac{4q-3}{2q}\Big( |u|_{2q}^{2q} + |v|_{2q}^{2q} +2\beta\int|u|^{q}|v|^{q}\Big) &
%\geq
%\end{eqnarray*}
%\todo[inline]{Completare.}

We conclude this section showing a further interesting property.

By Remark \ref{rem:positivita}, in polar form the ground state
$(\mathfrak u_{\beta}, \mathfrak v_{\beta})$ can be written as
\begin{equation}\label{eq:polar}
(\mathfrak u_{\beta}, \mathfrak v_{\beta}) =(\varrho_{\beta} \cos\vartheta_{\beta},\varrho_{\beta} 
\sin\vartheta_{\beta}),
\quad \varrho_{\beta}^2 = \mathfrak u_{\beta}^{2}+\mathfrak v_{\beta}^{2} >0, \quad \vartheta_{\beta}=\vartheta_{\beta}(x)\in[0,\pi/2].
\end{equation}

Note that whenever the ground state is vectorial, then $\vartheta_{\beta}\in (0,\pi/2)$, while
for semitrivial ground states it is $\vartheta_{\beta}\equiv0$ or $\vartheta_{\beta}\equiv\pi/2$. 

%\textcolor{red}{The pair $(\varrho_{\beta} \cos\theta_{\beta}, \varrho_{\beta} \sin\theta_{ \beta})$ is projectable on $%\mathcal M$ and indeed we have the following }
%but for our purpose it is not easy to deal with the above  form
%which involves the angular function $\vartheta_{\beta}$.

The next lemma shows how, starting from  $(\mathfrak u_{\beta}, \mathfrak v_{\beta})$
we can obtain  a convenient  ground state
%\textcolor{red}{at least for $\beta>0$ small enough,}\textcolor{blue}{for any $\beta>0$}
with the additional property of having  as angular coordinate a constant function $\theta_\beta$. \\
By Lemma \ref{unionefacili}, let $ y_{\beta}=\cos^2 \theta_\beta \in [0,1/2]$
be a  maximum point of $\mathfrak h_\beta$.
% Of course, choosing $1-y_\beta$ instead of $y_\beta$, we can argue in the same way.

%Indeed we have
\begin{Lem}\label{Mbeta}
%There exists $t_{\beta}>0$ such that
%$$(\mathfrak u_{\beta}, \mathfrak v_{\beta}) = (t_\beta^2 \varrho_\beta(t_\beta \cdot) \cos \theta_\beta,t_\beta^2 %\varrho_\beta(t_\beta \cdot) \sin \theta_\beta).$$
For  $\beta\geq0$, there exists $t_\beta > 0$ such that 
	$\gamma_{\varrho_{\beta}\cos\theta_{\beta},\varrho_{\beta}\sin\theta_{\beta}}(t_\beta)
	%(t_\beta^2 \varrho_\beta(t_\beta \cdot) \cos \theta_\beta,t_\beta^2 \varrho_\beta(t_\beta \cdot) \sin \theta_\beta)
	\in \mathcal{M}$
	and
	\begin{equation}\label{eq:gslevel}
	\mathfrak m_{\beta}
	%I_{\lambda,  \beta}(\mathfrak u_\beta, \mathfrak v_\beta)
	= I_{\lambda,  \beta}(\gamma_{\varrho_{\beta}\cos\theta_{\beta},\varrho_{\beta}\sin\theta_{\beta}}(t_\beta)
	%t_\beta^2 \varrho_\beta(t_\beta \cdot) \cos \theta_\beta,t_\beta^2 \varrho_\beta(t_\beta \cdot) \sin \theta_\beta
	).
	\end{equation}
In particular $\gamma_{\varrho_{\beta}\cos\theta_{\beta},\varrho_{\beta}\sin\theta_{\beta}}(t_\beta)
%(t_\beta^2 \varrho_\beta(t_\beta \cdot) \cos \theta_\beta,t_\beta^2 \varrho_\beta(t_\beta \cdot) \sin \theta_\beta)
$ is a  ground state solution.
\end{Lem}
\begin{proof}
The conclusion will be achieved showing that the {\em projection} of
$(\varrho_{\beta} \cos\theta_{\beta}, \varrho_{\beta} \sin\theta_{\beta})$ in $\mathcal M$  reaches the ground state level.\\
%, namely that there exists $t_\beta > 0$ such that 
%$$(t_\beta^2 \varrho_\beta(t_\beta \cdot) \cos \theta_\beta,t_\beta^2 \varrho_\beta(t_\beta \cdot) \sin \theta_\beta) \in %\mathcal{M}$$
%and
%\begin{equation}
%I_{\lambda,  \beta}(\mathfrak u_\beta, \mathfrak v_\beta)= I_{\lambda,  \beta}(t_\beta^2 \varrho_\beta(t_\beta \cdot) %\cos \theta_\beta,t_\beta^2 \varrho_\beta(t_\beta \cdot) \sin \theta_\beta).
%\end{equation}
Since $(\varrho_{\beta} \cos\theta_{\beta}, \varrho_{\beta} \sin\theta_{\beta}) \in \textrm{H}_{\textrm{r}}\setminus\{0\}$,
by Lemma \ref{Step0} there exists a unique $t_{\beta}>0$ such that
$$\gamma_{\varrho_{\beta}\cos\theta_{\beta},\varrho_{\beta}\sin\theta_{\beta}}(t_\beta)
%(t_\beta^2 \varrho_\beta(t_\beta \cdot) \cos \theta_\beta,t_\beta^2 \varrho_\beta(t_\beta \cdot) \sin \theta_\beta)
\in \mathcal{M}.$$
Let us show it  is at  the ground state level. Observe that
\[
\mathfrak m_{\beta}
%I_{\lambda,  \beta}(\mathfrak u_\beta, \mathfrak v_\beta)&
=
\inf_{(u,v)\in \textrm{H}_{\textrm{r}} \setminus\{0\}} \max_{t>0} I_{\lambda, \beta}(\gamma_{u,v}(t))
\leq
\max_{t>0} I_{\lambda, \beta}(
\gamma_{\varrho_{\beta}\cos\theta_{\beta},\varrho_{\beta}\sin\theta_{\beta}}(t)
%t^2\varrho_\beta(t \cdot) \cos \theta_\beta,t^2 \varrho_\beta(t \cdot) \sin \theta_\beta
)
=
I_{\lambda, \beta}(
\gamma_{\varrho_{\beta}\cos\theta_{\beta},\varrho_{\beta}\sin\theta_{\beta}}(t_\beta)
%t_\beta^2 \varrho_\beta(t_\beta \cdot) \cos \theta_\beta,t_\beta^2 \varrho_\beta(t_\beta \cdot) \sin \theta_\beta
).
\]
	Moreover, %since in virtue   of Lemma \ref{unionefacili}, 
	being $y_{\beta}=\cos^2 \theta_\beta $  a  maximum point of $\mathfrak h_\beta$ in $[0,1/2]$,
	\begin{equation*}\label{eq:serveilmassimo}
	\mathfrak u_{\beta}^{2q} +\mathfrak v_{\beta}^{2q} +2\beta \mathfrak u_{\beta}^{q} \mathfrak v_{\beta}^{q}
%	&={\color{blue}\varrho_{\beta}^{2q}} (\cos^{2q}\vartheta_{\beta} +\sin^{2q}\vartheta_{\beta} +2\beta \cos^{q}\vartheta_{\beta}\sin^{q}\vartheta_{\beta}) \\
	= \mathfrak h_{\beta}(\cos^{2}\vartheta_{\beta}) \varrho_{\beta}^{2q}\\
	\leq \mathfrak h_{\beta} (y_{\beta}) \varrho_{\beta}^{2q}.
	\end{equation*}
	Then, since 
	\begin{equation*}
\|\nabla \mathfrak u_{\beta} \|_{2}^{2} + \|\nabla \mathfrak  v_{\beta} \|_{2}^{2}
= \|\varrho_{\beta}\nabla \vartheta_{\beta}\|_2^{2} + \|\nabla \varrho_{\beta}\|_2^{2}\geq 
\| \nabla \varrho_{\beta}\|_{2}^{2},
\end{equation*}
 for every $t>0$,
	\begin{align*}
	I_{\lambda, \beta}(
	\gamma_{\varrho_{\beta}\cos\theta_{\beta},\varrho_{\beta}\sin\theta_{\beta}}(t)
	%t^2 \varrho_\beta(t \cdot) \cos \theta_\beta,t^2 \varrho_\beta(t \cdot) \sin \theta_\beta
	)
%	&=
%	\frac{t^{3}}{2}\|\nabla \varrho_{\beta}\|_{2}^{2}+\frac{t}{2}\|\varrho_{\beta}\|_{2}^{2}+\frac{\lambda}{4}t^{3}\int \phi_{\varrho_{\beta}}\varrho_{\beta}^{2}\\
%	&
%	-\frac{t^{4q-3}}{2q}\int  \varrho_{\beta}^{2q} \cos^{2q}\theta_{\beta}+
%	\varrho_{\beta}^{2q} \sin^{2q}\theta_{\beta}+2\beta \varrho_{\beta}^{2q} \cos^{q}\theta_{\beta} \sin^{q}\theta_{\beta}\\
	&=
	\frac{t^{3}}{2}\|\nabla \varrho_{\beta}\|_{2}^{2}+\frac{t}{2}\|\varrho_{\beta}\|_{2}^{2}+\frac{\lambda}{4}t^{3}\int \phi_{\varrho_{\beta}}\varrho_{\beta}^{2}
	-\frac{t^{4q-3}}{2q} \mathfrak h_{\beta} (y_{\beta}) \|\varrho_{\beta}\|_{2q}^{2q}\\
	&\leq
	\frac{t^{3}}{2}\|\nabla \varrho_{\beta}\|_{2}^{2}+\frac{t}{2}\|\varrho_{\beta}\|_{2}^{2}+\frac{\lambda}{4}t^{3}\int \phi_{\varrho_{\beta}}\varrho_{\beta}^{2}
	-\frac{t^{4q-3}}{2q}\int \left( \mathfrak u_{\beta}^{2q} +\mathfrak v_{\beta}^{2q} +2\beta \mathfrak u_{\beta}^{q} \mathfrak v_{\beta}^{q}\right)\\
	&\leq
	I_{ \lambda, \beta}(\gamma_{\mathfrak u_{\beta}, \mathfrak v_{\beta}}(t)).
	\end{align*}
	Thus, by Lemma \ref{Step0},
	%\begin{equation}\label{maxont}
	%I_{ \beta}(t^2\mathfrak u_\beta(t \cdot),t^2 \mathfrak v_\beta(t \cdot ))
	%\geq 
	%I_{\beta}(t^2 \varrho_\beta(t \cdot) \cos \theta_\beta,t^2 \varrho_\beta(t \cdot) \sin \theta_\beta)
	%\end{equation}
	%and then
	$$
	I_{\lambda,\beta}(
	\gamma_{\varrho_{\beta}\cos\theta_{\beta},\varrho_{\beta}\sin\theta_{\beta}}(t_\beta)
	%t_{\beta}^2 \varrho_\beta(t_{\beta} \cdot) \cos \theta_\beta,t_{\beta}^2 \varrho_\beta(t_{\beta} \cdot) \sin \theta_\beta
	)
	\leq I_{ \lambda,\beta}(\gamma_{\mathfrak u_{\beta}, \mathfrak v_{\beta}}(t_{\beta})
	%^2\mathfrak u_\beta(t_{\beta} \cdot),t_{\beta}^2 \mathfrak v_\beta(t_{\beta} \cdot )
	)
	\leq \mathfrak m_{\beta}, %I_{\lambda,\beta}(\mathfrak u_{\beta}, \mathfrak v_{\beta}),
		$$
	concluding the proof.	
%	and we know that
%	\begin{equation}\label{eq:maledetta}
%	\int \varrho_{\beta}^{2q} = \int \left(\mathfrak u^{2}_{\beta} +\mathfrak v^{2}_{\beta}\right)^{q}\geq
%	\int \left( \mathfrak u_{\beta}^{2q} +\mathfrak v_{\beta}^{2q} +2\beta \mathfrak u_{\beta}^{q} \mathfrak v_{\beta}%^{q}\right)
%	\end{equation}
	%{\bf Nota che se $\mathfrak u_{\beta}= \lambda \mathfrak v_{\beta}$  la disuguaglianza \eqref{eq:maledetta}
	%sopra vale per $\beta\leq 2^{q-1}-1$!! .}
%	 and using that....
%	Thus, passing to the maximum on $t>0$ in (\ref{maxont}), we conclude.
\end{proof}

%\begin{lemma}\ref{lem:altrolemma}
%\begin{itemize}
%\item[1.] $q=2$: for any $\beta\leq1$ it is $(t^{2}+1)^{q} \geq t^{2q}+1+2\beta t^{q}$.  \\
%OBVIOUS.
%\item[2.] $3/2<q<2$ : for $\beta$ small, $(t^{2}+1)^{q} \geq t^{2q}+1+2\beta t^{q}$.\\
%OBVIOUS IF $\beta=0$: study $f(t) = (t^{2}+1)^{q} - t^{2q}-1$.
%
%I hope that for small $\beta$ it is also true....
%\end{itemize}
%\end{lemma}

\begin{Rem}\label{rem:generale}
Whenever  $\max_{[0,1]}\mathfrak h_{\beta}>1$, then $y_{\beta}\in(0,1/2]$
and  Lemma \ref{Mbeta} gives a vectorial ground state.
 \end{Rem}

\section{The case $\beta=0$}\label{sec:bzero}

In this section we prove  item \eqref{th:gsi} of Theorem \ref{th:gs}: if $\beta=0$,  each  ground state solution $(\mathfrak{u}_0,\mathfrak{v}_0)$ of \eqref{system} is semitrivial.

Consider system \eqref{system} for $\beta=0$, namely
\begin{equation}
\label{system0}
\begin{cases}
-\Delta u + u + \lambda \phi_{u,v} u = |u|^{2q-2} u \\
-\Delta v + v + \lambda \phi_{u,v} v = |v|^{2q-2} v
\end{cases}
\text{in }\mathbb{R}^3.
\end{equation}
Of course %we have that 
$(\mathfrak{w},0)$ and $(0,\mathfrak{w})$, where $\mathfrak{w}$ is a ground state of \eqref{eq:ruiz} obtained in \cite{RuizJFA} (see also Section \ref{sec2}), are solutions  of \eqref{system0}. %Moreover, 
Since for every $u\in H^1_{\textrm{r}}(\mathbb{R}^3)$, we have that 
$$\mathcal{I}_{\lambda,0}(u)=I_{\lambda,0}(u,0)=I_{\lambda,0}(0,u),$$
then, by \eqref{gslRuiz},
\[
\mathfrak n
=\mathcal{I}_{\lambda,0}(\mathfrak{w})
=I_{\lambda,0}(\mathfrak{w},0)=I_{\lambda,0}(0,\mathfrak{w})
=\inf_{u\in H_{\textrm r}^1(\mathbb{R}^3)\setminus\{0\}} \max_{t>0} I_{\lambda,0} (\zeta_u(t),0),
\]
where the path $\zeta_u$ has been defined in \eqref{zeta}.\\
Moreover
\begin{equation}
\label{gslleqgslR}
\mathfrak{m}_0 \leq \mathfrak{n}.
%I_{\lambda,0}(\mathfrak{u},\mathfrak{v})\leq \mathcal{I}_{\lambda,0}(\mathfrak{w}).
\end{equation}
%\textcolor{teal}{The main result of this section is the following.
%\begin{Th}\label{Thmbeta0}
%If $\beta=0$, then the  ground state solution of \eqref{system} is semitrivial.
%\end{Th}
%}

First we prove that the ground state level of the two variables functional $I_{\lambda,0}$ is the same of the ground state level of the one variable functional $\mathcal{I}_{\lambda,0}$.
\begin{Lem}\label{Lembeta0}
$\mathfrak{m}_0 = \mathfrak{n}$.
%$I_{\lambda,0}(\mathfrak{u},\mathfrak{v})=\mathcal{I}_{\lambda,0}(\mathfrak{w})$.
\end{Lem}
\begin{proof}
%First observe that, if $(u,v)\in{\rm H}_r\setminus\{0\}$, then
%\begin{equation}
%\label{eq:Ilambdazero}
%\begin{split}
%I_{0}(\gamma_{u,v}(t))
%&=
%\frac{t^{3}}{2} ( \|\nabla u\|_{2}^{2}+ \|\nabla v\|_{2}^{2} )
%+\frac{t}{2} ( \|u\|_{2}^{2} +\|v\|_{2}^{2}) \\
%&\quad
%+ \frac{\lambda t^{3}}{2} \int (u^{2}+v^{2}) \phi_{u,v}
%-\frac{t^{4q-3}}{2q} ( \|u\|_{2q}^{2q} + \|v\|_{2q}^{2q})\\
%&=
%I_{0}(\gamma_{|u|,|v|}(t)).
%\end{split}
%\end{equation}
If we use the {\sl polar coordinates} for the couples $(u,v)$, namely we write
$$(u, v) = (\varrho \cos\vartheta, \varrho\sin\vartheta)
\text{ where } \varrho ^2 = u^{2}+v^{2} \text{ and } \vartheta=\vartheta(x) \in [0,2\pi],$$
we have that
\begin{equation*}
\|\nabla u \|_{2}^{2} + \|\nabla v \|_{2}^{2}
= \|\varrho\nabla \vartheta\|_2^{2} + \|\nabla \varrho\|_2^{2}
\end{equation*}
and, by (\ref{facilei}) in Lemma \ref{unionefacili},
\begin{equation*}\label{eq:disub=0}
\|u\|_{2q}^{2q} +\|v\|_{2q}^{2q} 
= \int \varrho^{2q} \left( \cos^{2q}\vartheta + \sin^{2q}\vartheta\right)
\leq \| \varrho \|_{2q}^{2q}.
\end{equation*}
Then, for every $t>0$,
\begin{equation}
\label{eq:bo1}
I_{\lambda,0}(\gamma_{u,v}(t))
\geq
\frac{t^{3}}{2} \|\varrho \nabla \vartheta\|_2^{2} 
+ \frac{t^{3}}{2} \|\nabla \varrho \|_2^{2}
+ \frac{t}{2} \|\varrho \|_2^{2}
+ \frac{\lambda }{4}t^{3}\int \phi_\varrho \varrho^{2}
-\frac{t^{4q-3}}{2q} \| \varrho \|_{2q}^{2q}
\geq
\mathcal{I}_{\lambda,0}(\zeta_\varrho(t)). %= I_{0}(\zeta_\varrho(t), 0) = I_{0}(0, \zeta_\varrho(t))
\end{equation}
Hence, \eqref{eq:bo1}, \eqref{gslRuiz}, and \eqref{gslleqgslR} imply
\[
\mathfrak{m}_0
=\inf_{(u,v)\in{\rm H}_{\textrm r}\setminus\{0\}} \max_{t>0} I_{\lambda,0} (\gamma_{u,v}(t))
\geq
\inf_{\varrho \in H_{\textrm r}^1(\mathbb{R}^3)\setminus\{0\}} \max_{t>0} I_{\lambda,0} (\zeta_\varrho(t),0) =
\mathfrak{n},
\]
concluding the proof.
\end{proof}

Now we are ready to show the main goal of this section.
 
\begin{proof}[Proof of  \eqref{th:gsi} of Theorem \ref{th:gs}.]
Assume by contradiction that there exists a vectorial ground state $( \overline u, \overline v)$. 
By Remark \ref{rem:positivita}, without loss of generality we can assume that $\overline u, \overline v>0$.
Thus, using as before the {\em polar coordinates}, we can write
	$$(\overline u,\overline v)
	= (\overline \varrho \cos\overline\vartheta, 
	\overline \varrho \sin\overline\vartheta),
	\text{ with }
	\overline \varrho^2 =\overline u^{2} +\overline v^{2}
	\text{ and } \overline\vartheta=\overline\vartheta(x)\in(0,\pi/2).$$
	Then, using (\ref{facilei}) in Lemma \ref{unionefacili}, we have that $\cos^{2q}\overline \vartheta+\sin^{2q}\overline\vartheta<1$, and so by \eqref{formulaStep0} and arguing as in \eqref{eq:bo1},
we get that, for all $t>0$, 
	\[
	\mathfrak{m}_0
	\geq I_{\lambda,0}(\gamma_{\overline u, \overline v}(t))
	>I_{\lambda,0}(\zeta_{\overline \varrho}(t),0).
	\]
	Then
	$$
	\mathfrak{m}_0
	> \max_{t>0} I_{\lambda,0}(\zeta_{\overline \varrho}(t),0)
	\geq \inf_{\varrho\in H_{\textrm r}^1(\mathbb{R}^3)\setminus\{0\}} \max_{t>0} I_{\lambda,0}(\zeta_\varrho(t),0)
	= \mathfrak{n},
	$$
	which is a contradiction with Lemma \ref{Lembeta0}.
\end{proof}

\section{The case $\beta>0$ and small}\label{sec:POSITIVESMALL}

%In this section we consider system \eqref{system} for $\beta>0$ small.
%In this case, the ground states are semitrivial or vectorial depending on the values of $q$. Indeed we have the %following
%\todo[inline]{Gaetano: controllare se il punto b. vale per ogni beta}
%\textcolor{teal}{
%\begin{Th}\label{Thmbetaqgrande}
%	If $\beta$ is sufficiently small, then:
%	\begin{enumerate}[label=(\alph{*}), ref=\alph{*}]
%	\item \label{1} for $q\in [2,3)$, the ground states solution of \eqref{system} are semitrivial;
%	\item \label{2} for $q\in (3/2,2)$, the ground states solution of \eqref{system} are vectorial.
%	\end{enumerate}
%\end{Th}
%}

%To prove the result some preliminaries are in order.
%As in the previous Section, let us start with a technical result,
%which indeed is like Lemma \ref{lem:facile}
%up to a ``small'' perturbation.

In this section we consider
\[
\beta\in
\left\{
\begin{array}{lll}
	(0,2^{q-1}-1)& \text{ for }q\in[2,3)&\text{ as in  (\ref{th:gsiib}) of Theorem \ref{th:gs}},\\
	(0,q-1) & \text{ for }q\in(3/2,2)&\text{ as in  (\ref{th:gsiiib}) of Theorem \ref{th:gs}.}
\end{array}
\right.
\]
%and  prove (\ref{th:gsiib}) and (\ref{th:gsiiib}) of Theorem \ref{th:gs}.

Let us start with the proof of item (\ref{th:gsiib}) of Theorem \ref{th:gs}.

Preliminarily, as in Section \ref{sec:bzero}, we prove
\begin{lemma}\label{Lemma51}
If $\beta\in(0,2^{q-1}-1]$ and $q\in[2,3)$, then $\mathfrak m_{\beta}= \mathfrak n$.
\end{lemma}
\begin{proof}
%\textcolor{red}{We argue as in Section \ref{subsec:bzero} 
%where now \eqref{eq:Ilambdazero}
%is replaced by
%\begin{eqnarray*}%\label{eq:Ilambdabeta}
%	I_{\beta}(\gamma_{u,v}(t)) &=& \frac{t^{3}}{2} \left( \|\nabla u\|_{2}^{2} + \|\nabla v\|_{2}^{2} \right)+\frac{t}{2} \left( \|u\|_{2}^{2} +\|v\|_{2}^{2}\right) \\
%	&+&\frac{\lambda t^{3}}{2} \int  \left(\frac{1}{|\cdot|} *(u^{2}+v^{2}) \right) (u^{2}+v^{2}) -\frac{t^{4q-3}}{2q}
%	\left( \|u\|_{2q}^{2q} + \|v\|_{2q}^{2q} +2\beta\int |u|^{q} |v|^{q}\right) \nonumber.
%\end{eqnarray*}
%}
Since for any $u\in H_{\textrm{r}}^{1}(\mathbb R^{3})$ it holds
\begin{equation}\label{eq:COVID}
	\mathcal I_{\lambda,0} (u) = I_{\lambda,\beta}( u,0)= I_{\lambda,\beta}(0,u),
\end{equation}
then 
\[
\mathfrak n
=\mathcal{I}_{\lambda,0}(\mathfrak{w})
%=I_{\lambda,\beta}(\mathfrak{w},0)=I_{\lambda,\beta}(0,\mathfrak{w})
=\inf_{u\in H_{\textrm r}^1(\mathbb{R}^3)\setminus\{0\}} \max_{t>0} I_{\lambda,\beta} (\zeta_u(t),0)
\geq \inf_{(u,v)\in \textrm H_{\textrm r} \setminus\{0\}} \max_{t>0} I_{\lambda,\beta} ( \gamma_{u,v}(t))
= \mathfrak m_{\beta}.
\]
Moreover, introducing the  polar coordinates as in Lemma \ref{Lembeta0}
and using (\ref{aiiib}) and  (\ref{biiib}) of (\ref{facileiiib}) in Lemma \ref{unionefacili} we get
$$
\|u\|_{2q}^{2q} + \|v\|_{2q}^{2q} +2\beta\int |u|^{q} |v|^{q}
=
\int \varrho^{2q}\left(\cos^{2q} \vartheta +\sin^{2q}\vartheta
+2\beta |\cos\vartheta|^{q}(1-\sin^{2}\vartheta)^{q/2} \right)
\leq
\| \varrho\|_{2q}^{2q}.
$$
Thus, arguing as in Lemma \ref{Lembeta0},  we arrive at
%\begin{equation}\label{eq:}
$I_{\lambda, \beta}(\gamma_{u,v}(t)) \geq \mathcal I_{\lambda,0}(\zeta_{\rho}(t))$
%\end{equation}
and so $\mathfrak m_{\beta}\geq \mathfrak n$.
\end{proof}
As an immediate consequence we can conclude as follows.
\begin{proof}[Proof of  (\ref{th:gsiib}) in Theorem \ref{th:gs}]
Hence for $\beta\in (0,2^{q-1}-1)$, the proof is completely analogous to that one of item \eqref{th:gsi} of Theorem \ref{th:gs}, using  (\ref{aiiib}) of (\ref{facileiiib}) in Lemma \ref{unionefacili} instead if (\ref{facilei}).
\end{proof}
%, taking into account  Lemma \ref{lem:facile2}, we have
%\begin{eqnarray}\label{eq:bo2}
%I_{\beta}(\gamma(t)) &=&  \frac{t^{3}}{2}\int |\nabla \theta|^{2} \varrho^{2} +\frac{t^{3}}{2}\int |\nabla \varrho|^{2} 
%+\frac{\lambda t^{3}}{4}\int \phi_{\varrho}  \varrho^{2} \\
%&-&\frac{t}{2}\int \varrho^{2} -\frac{t^{4q-3}}{2q}\int \varrho^{2q}\left(\cos^{2q}\theta+\sin^{2q}\theta+2\beta
%\cos^{q/2}\theta(1-\sin^{2}\theta )^{q/2}\right) \nonumber \\
%&\geq& \frac{t^{3}}{2}\int |\nabla \theta|^{2} \varrho^{2} +\int |\nabla \varrho|^{2} 
%+\frac{\lambda t^{3}}{2}\int \left( \frac{1}{|\cdot|}*\varrho^{2}  \right) \varrho^{2}
%-\frac{t}{2}\int \varrho^{2} -\frac{t^{4q-3}}{2q}\int \varrho^{2q} \nonumber \\
%&=&I_{\beta}(t^{2}\varrho(t\cdot), 0) \ \   \left(= I_{\beta}(0, t^{2}\varrho(t\cdot)) \right) \nonumber.
%\end{eqnarray}
%Then the  argument in Subsection \ref{subsec:bzero}
%still work with the same proof, obtaining that

Let us address now \eqref{th:gsiiib} of Theorem \ref{th:gs}.
% namely  the case $q\in(3/2,2)$ and $\beta\in (0,q-1)$. 
We first show that
the ground state $(\mathfrak u_{\beta}, \mathfrak v_{\beta})$ is vectorial  and then  that it converges to a semitrivial solution as $\beta\to 0^{+}$.

First we show a preliminary property that we will use also in the next section, since the conclusion holds whenever 
there exists a point where  $\mathfrak h_{\beta}$ introduced in Lemma \ref{unionefacili} is greater than $1$.
%(and, consequently,  the maximum of the function

\begin{lemma}\label{Lemma52}
	If $\beta\in (0, q-1)$ and $q\in(3/2,2)$, then
	\begin{equation}
		\label{upperbound}
		%I_{\lambda,\beta}(\mathfrak u, \mathfrak v)
		\mathfrak m_{\beta}<\mathfrak n.
		% \mathcal I_{\lambda,0}(\mathfrak w)
	\end{equation}
\end{lemma}
\begin{proof}
In virtue of  (\ref{facileiib}) of Lemma \ref{unionefacili}, let $\theta_{\beta}\in (0,\pi/2)$
be defined by $y_{\beta} = \cos^{2} \theta_{\beta}\in(0,1/2)$
and consider the vectorial function 
$$(\widetilde u,\widetilde v):=(\mathfrak w \cos\theta_{\beta} , \mathfrak w\sin \theta_{\beta})\in \textrm H_{ \textrm r}\setminus\{0\}.$$
A simple computation shows that
\begin{align*}
	\|\nabla \widetilde u\|_{2}^{2}  +\|\nabla \widetilde v\|_{2}^{2}
	&=
	\|\nabla \mathfrak w\|_{2}^{2} \\
	\|\widetilde u\|_{2}^{2}+ \|\widetilde v\|_{2}^{2}
	&=
	\|\mathfrak w\|_{2}^{2},\\
	\int (\widetilde u^{2}+\widetilde v^{2})\phi_{\widetilde u,\widetilde v}
	&=
	\int \phi_{\mathfrak w} \mathfrak w^{2},\\
	\|\widetilde u\|_{2q}^{2q} + \|\widetilde v\|_{2q}^{2q} +2\beta\int |\widetilde u|^{q} |\widetilde v|^{q}
	&=
	\left( \cos^{2q}\theta_{\beta}+\sin ^{2q} \theta_{\beta} +2\beta \cos^{q}\theta_{\beta} 
	\sin^{q}\theta_{\beta}\right)
	\|\mathfrak w\|_{2q}^{2q}\\
	&=
	\left( y_{\beta}^{q} +(1-y_{\beta})^{q} +2\beta y_{\beta}^{q/2} (1-y_{\beta})^{q/2} \right)
	\|\mathfrak w\|_{2q}^{2q}\\
	&>
	\| \mathfrak w\|_{2q}^{2q},
\end{align*}
where the last inequality is due again to (\ref{facileiib}) of Lemma \ref{unionefacili}.\\
Consequently, recalling \eqref{eq:COVID}, for any $t>0$ we have
\begin{equation*}
	\begin{split}
		\mathcal I_{\lambda,0}(\zeta_{\mathfrak w}(t))
		= I_{\lambda, \beta}(\gamma_{\mathfrak w,0}(t)) &=\frac{t^{3}}{2} \| \nabla \mathfrak w\|_{2}^{2}  +\frac{t}{2}\|\mathfrak w\|^{2}_{2}+\frac{\lambda}{4}t^{3}\int \phi_{\mathfrak w} \mathfrak w^{2} -\frac{t^{4q-3}}{2q}\|\mathfrak w\|_{2q}^{2q}\\
		&>
		\frac{t^3}{2} (\|\nabla \widetilde u \|_2^2 + \|\nabla \widetilde v \|_2^2)
		+ \frac{t}{2} (\| \widetilde u \|_2^2 + \| \widetilde v \|_2^2)
		+ \frac{\lambda}{4} t^3 \int \phi_{\widetilde u,\widetilde v} (\widetilde u^2+\widetilde v^2) \\
		&\quad - \frac{t^{4q-3}}{2q}\left( \|\widetilde u\|_{2q}^{2q}+\|\widetilde v\|_{2q}^{2q} +2\beta\int |\widetilde u|^q |\widetilde v|^q \right)\\
		&=I_{\lambda,\beta}(\gamma_{\widetilde u,\widetilde v}(t)).
	\end{split}
\end{equation*}
Passing to the maximum on $t>0$, since both maxima are achieved,
and  recalling that $t\mapsto 
\mathcal I_{\lambda,0}(\zeta_{\mathfrak w}(t))$
achieves its maximum in $t=1$ being $\mathfrak w\in \mathcal N^{\lambda}$ (see \eqref{eq:NRuiz}), we can write
\[
\mathfrak n =\mathcal I_{\lambda,0}(\mathfrak w)  %\inf_{u\in H^{1}_{r}(\mathbb R^{3})\setminus\{0\}}\max_{t>0}  \mathcal I^{\lambda}(\gamma_{u}(t))
>  \max_{t>0} I_{\lambda,\beta}(\gamma_{\widetilde u,\widetilde v}(t)) \geq
\inf_{(u,v)\in \textrm H_{\textrm r}\setminus\{0\}} \max_{t>0} I_{\lambda,\beta}(\gamma_{u,v}(t)) =
\mathfrak m_{\beta}
%I_{\lambda,\beta}(\mathfrak u, \mathfrak v)
\]
concluding the proof.
\end{proof}

We point out that, in contrast to the proof of Lemma  \ref{Mbeta},  
where we need to take exactly  $y_{\beta}$, here in Lemma \ref{Lemma52}
it is enough to take an arbitrary point where $\mathfrak h_\beta$
is greater than one.

As an immediate consequence of Lemma \ref{Lemma52} we have
\begin{proof}[Proof of \eqref{th:gsiiib} in Theorem \ref{th:gs} (vectorial ground state)]
If the ground state were for instance  of type $(\mathfrak u_{\beta}, 0)$, then, recalling \eqref{eq:COVID}, 
	we would have
	$$\mathfrak{m}_\beta
		= I_{\lambda,\beta}(\mathfrak u_{\beta}, 0 )
	= \mathcal I_{\lambda,0} (\mathfrak u_{\beta}) 
	\geq  \mathfrak{n}$$
	contradicting \eqref{upperbound}.
	\end{proof}

As observed in Remark \ref{rem:generale}, by Lemma \ref{Mbeta}, if
  $ y_{\beta}=\cos^2 \theta_\beta \in (0,1/2]$
is the maximum point of $\mathfrak h_\beta$, we have
\begin{Cor}\label{cor:53}
If $\beta\in(0,q-1)$ and $q\in (3/2,2)$, then there exists $t_\beta > 0$ such that
$\gamma_{\varrho_{\beta}\cos\theta_{\beta},\varrho_{\beta}\sin\theta_{\beta}}(t_\beta)$
is a vectorial ground state.
\end{Cor}

%Let $\overline u>0$ be the ground state .... Since
%\begin{equation}\label{eq:}
%I_{\beta}(\overline u, 0)=c^{\textrm{semitrivial}} >
%I_{\beta}(t^{2} \overline u(t\cdot ) \cos\theta_{0}, t^{2} \overline u(t\cdot )\sin\theta_{0})\geq
%c_{\lambda} \min \max_{t>0} I_{\lambda,\beta}(\gamma (t))
%\end{equation}

Now we show the {\em asymptotic behavior}   as $\beta \to 0^+$ of the vectorial  ground state solutions found in (\ref{th:gsiiib}) of Theorem \ref{th:gs}.
As in the proof of \eqref{th:gsi} of Theorem \ref{th:gs}, we assume without loss of generality that 
$\mathfrak u_{\beta}$ and $\mathfrak v_{\beta}$ are positive.

%\textcolor{teal}{
%\begin{Th}\label{th:semitriviale}
%For $\beta\to 0$ the ground states solutions $(\mathfrak u_{\beta}, \mathfrak v_{\beta})$
%of \eqref{system}
%strongly converge to a semitrivial solution
%whose nontrivial component is a radial ground state solution of \eqref{eq:singolaRuiz}. 
%\end{Th}
%}
%

%with $\beta$ small (\textcolor{teal}{rivedere il small})
%We show that the weak limit is a semitrivial ground state solution of \eqref{eq:singolaRuiz}. Namely
%the main result here is the following.

Arguing as in Step 1 and Step 2 of the Proof of Theorem \ref{th:gs}, we can get for $\beta>0$ in a bounded set a uniform lower bound for the ground states levels $\mathfrak m_{\beta}$.\\
However, in what follows, we give an estimate of such a lower bound depending on the energy level of the ground state $\mathfrak{g}$ of
\[
-\Delta u + u   = |u|^{2q-2} u 
\text{ in }\mathbb{R}^3
\]
(see e.g. \cite{W}).\\
%\textcolor{red}{
%\begin{Lem}\label{upperbound}
%	If $\beta > 0$, then $I_{\lambda, \beta}(u_\beta, v_\beta) \leq I_\lambda(u_0)$.
%\end{Lem}
%\begin{proof}
%	Consider $(u_0 \cos \theta_\beta, u_0 \sin \theta_\beta)$,
%	where $\cos^2 \theta_\beta = y_\beta \in (0,1/2) $ is the unique maximum point of $h_\beta (y)$ for $y\in(0,1/2)$. By Lemma \ref{} there exists a unique $t_\beta > 0$ such that \\
%	$(t_\beta^2 u_0(t_\beta \cdot) \cos \theta_\beta,t_\beta^2 u_0(t_\beta \cdot) \sin \theta_\beta) \in \mathcal{M}$.
%	Since $h_{\beta}(y_{\beta}) \geq 1$, then
%	\begin{align*}
%	c_{\lambda,\beta} &:= I_{\lambda, \beta}(u_\beta,v_\beta)=\inf_{\mathcal M} I_{\lambda,\beta}\\
%	&\leq 
%	I_{\lambda, \beta}(t_\beta^2 u_0(t_\beta \cdot) \cos \theta_\beta,t_\beta^2 u_0(t_\beta \cdot) \sin \theta_\beta)\\
%	&=
%	\frac{t_\beta^2}{2} \|\nabla u_0 \|_2^2
%	+ \frac{t_\beta}{2} \| u_0 \|_2^2
%	+ \frac{\lambda}{2} {t_\beta^3} \int \left(\frac{1}{|\cdot|}*u_0^2\right) u_0^2
%	-\frac{{t_\beta}^{4q-3}}{2q} h_\beta(y_\beta) \|u_0\|_{2q}^{2q}\\
%	&\leq 
%	I_{\lambda}(t_{\beta}^{2}u_0(t_\beta \cdot))\\
%	&\leq I_{\lambda}(u_0),
%	\end{align*}
%	where this last term does not depend on $\beta$.
%\end{proof}
%}
To this aim let us introduce  another limit problem which 
will be useful for our purpose: system \eqref{system} with $\lambda=0$, namely
\begin{equation} \label{eq:systemMandel}
\begin{cases}
-\Delta u + u   = |u|^{2q-2} u + \beta |v|^q |u|^{q-2} u \medskip \\
-\Delta v + v  = |v|^{2q-2} v + \beta |u|^q |v|^{q-2} v
\end{cases}
\text{in }\mathbb{R}^3.
\end{equation}
Let 
$(\widehat{\mathfrak{u}}_{\beta}, \widehat{\mathfrak v}_{\beta} )\in \textrm{H}_{\textrm r}$ be the vectorial, positive and radial ground state solution,  see \cite[Corollary 1]{Mandel},
which exists for any $\beta>0$ and $q\in (3/2, 2)$.
In our notations, the  energy functional related to \eqref{eq:systemMandel} is $I_{0,\beta}$.
Since
$$I_{0, \beta}(\gamma_{\widehat{\mathfrak u}_{\beta}, \widehat{\mathfrak v}_{\beta}}(t))\to -\infty\quad \text{as } \
t\to+\infty,$$
 there exists $a_{\beta}>0$ such that 	
 %$I_{0, \beta}(\gamma_{\mathfrak u_{\beta}, \mathfrak v_{\beta}}(a_{\beta}))<0$ and then
	$$\gamma_{\widehat{\mathfrak u}_{\beta}, \widehat{\mathfrak v}_{\beta}} \in \Gamma_{\beta}: = \left\{\eta \in C([0,a_{\beta}], \textrm{H}_{\textrm r}): I_{0,\beta}(\eta(0))=0, I_{0,\beta}(\eta(a_{\beta}))<0) \right\}
	$$
	and we have the usual minimax characterisation of the ground state
%\textcolor{red}{	Moreover by \cite{}
%	\todo[inline]{referenza}}
\begin{equation}\label{eq:carattMandel}
\inf_{\eta \in \Gamma_{\beta}}\max_{t>0} I_{0, \beta}(\eta(t)) = I_{0, \beta}(\widehat{\mathfrak u}_\beta ,\widehat{\mathfrak v}_\beta ),
\end{equation}
see e.g. \cite[Lemma 3.2]{MMP}.\\
The next lemma allows us to get the desired  lower bound.

\begin{Lem}\label{lowerbound}
	If $0< \beta \leq 1$, then
	%there exists $\underbar c >0$  such that $I_{\lambda, \beta}(u_\beta, v_\beta) \geq \underbar c$.
	$$ %I_{\lambda, \beta}(\mathfrak u_\beta, \mathfrak v_\beta) 
	\mathfrak m_{\beta}
	\geq  2^{\frac{q-2}{q-1}}I_{0,\beta}(\mathfrak g, 0)
	= 2^{\frac{q-2}{q-1}}\mathcal I _{0,0} (\mathfrak g).$$
\end{Lem}
\begin{proof}	
	 By Lemma \ref{Step0} and \eqref{eq:carattMandel} it holds
	 \[
	 	\mathfrak m_{\beta}
	 	= \max_{t>0} I_{\lambda, \beta}(\gamma_{\mathfrak u_{\beta}, \mathfrak v_{\beta}}(t))
	 	\geq \max_{t>0} I_{0, \beta}(\gamma_{\mathfrak u_{\beta}, \mathfrak v_{\beta}}(t))
	 	\geq  \inf_{\eta \in \Gamma_{\beta}}\max_{t>0} I_{0, \beta}(\eta(t)) 
	 	= I_{0, \beta}(\widehat{\mathfrak u}_\beta ,\widehat{\mathfrak v}_\beta ).
	 \]
%	\begin{align*}
%%	I_{\lambda, \beta}(\mathfrak u_\beta, \mathfrak v_\beta)
%	\mathfrak m_{\beta}
%	&= \max_{t>0} I_{\lambda, \beta}(t^2\mathfrak u_\beta(t \cdot) ,t^2 \mathfrak v_\beta(t \cdot) )\\
%	&\geq \max_{t>0} I_{0, \beta}(t^2\mathfrak u_\beta(t \cdot) ,t^2 \mathfrak v_\beta(t \cdot) )\\
%	&\geq  \inf_{\eta \in \Gamma_{\beta}}\max_{t>0} I_{0, \beta}(\eta(t)) \\
%	&= I_{0, \beta}(\widehat{\mathfrak u}_\beta ,\widehat{\mathfrak v}_\beta ).
%	\end{align*}
	%	where $(\tilde u_\beta ,\tilde v_\beta )$ is the ground state solution for the local problem 
	%	\eqref{eq:systemMandel} with $\beta > 0$, which is positive and radial, as in \cite{MMP2006,Mandel}. 
	On the other hand, by  \cite[Proof of Lemma 4]{Mandel} we know
	%if $\mathfrak w$ is a positive ground state solution of the scalar problem \eqref{eq:singolaRuiz}, % with $\beta 
	%=0$, then
	\begin{equation*}\label{mandel}
	%I_{\lambda, \beta}(\widehat{\mathfrak u}_\beta ,\widehat{\mathfrak v}_\beta )\geq 
	I_{0, \beta}(\widehat{\mathfrak u}_\beta ,\widehat{\mathfrak v}_\beta )
	%\textcolor{red}{= I_{0, \beta}(\tilde u_\beta ,\tilde v_\beta )}
	\geq 
	\left(\inf_{k>0}\displaystyle{\frac{(1+k^{2})^q}{1+k^{2q}+2\beta k^q}}\right)^{\frac{1}{q-1}} I_{0, \beta}(\mathfrak g,0).
	\end{equation*}
	If we set $$\xi_\beta(k):= \displaystyle{\frac{(1+k^{2})^q}{1+k^{2q}+2\beta k^q}},$$ we have that, if $\beta_1< \beta_2$, then, for every $k>0$,
	$\xi_{\beta_1}(k) > \xi_{\beta_2}(k)$. 
	Hence, if $\beta\in(0,1]$, 
	%taking into account that $\xi_{1}(1) = 2^{q-2}$ and that $\xi_{1}'(k)<0$ for $k\in(0,1)$ {\color{red}(being $q\leq2$)},\todo{Forse non serve.}
	it holds
	\[
	\xi_\beta(k) \geq 
	%\displaystyle{\frac{(1+k^{2})^q}{(1+k^{q})^2}}=
	 \xi_1(k) \geq  \xi_{1}(1)=	2^{q-2}
	\]
	and the conclusion follows.
	\end{proof}
	
To prove the asymptotic behavior of the ground states $(\mathfrak u_{\beta}, \mathfrak v_{\beta})$, 
we will show 
first the asymptotic behaviour of the ground state levels.
However in order to do that, we will work
with another family of ground states different from $(\mathfrak u_{\beta}, \mathfrak v_{\beta})$.

%\textcolor{red}The next result states that the vectorial ground state solutions tends to a semitrivial solution for
%$\beta\to 0$.
%\begin{Lem}\label{limitsolution}
%	If $\beta \to 0$, then $(\mathfrak u_\beta,\mathfrak  v_\beta) \rightharpoonup (\widetilde u, 0)$, where $ %(\widetilde u, 0)$ is a semitrivial solution of the nonlocal problem \eqref{eq:singolaRuiz}.
%\end{Lem}
Now we are ready to prove the {\em asymptotic behavior} as $\beta \to 0^+$ of the ground state levels.

\begin{Prop}\label{th:tendearuiz}
	If $q\in (3/2, 2)$, as $\beta\to 0^+$, the  family of radial ground state solutions of \eqref{system} found in Lemma \ref{Mbeta}  converges in ${\rm H}_{\emph{r}}$  to a semitrivial solution of \eqref{system}, whose nontrivial component is a radial ground state solution of \eqref{eq:ruiz}.
	Moreover 
	\begin{equation}\label{eq:mn}
\lim_{\beta\to 0^+}	\mathfrak m_\beta = \mathfrak n.
	\end{equation}
%\textcolor{red}{	where $\mathfrak n$ is defined in \eqref{eq:n}.}
\end{Prop}
\begin{proof}
Using the notations of Lemma \ref{Mbeta}, let us set
$${\rm u}_{\beta} :=\widetilde \varrho_\beta\cos\theta_{\beta}, 
\quad {\rm v}_{\beta} := \widetilde \varrho_\beta\sin\theta_{\beta},
\quad \widetilde \varrho_\beta := \zeta_{\rho_\beta}(t_\beta)=t_\beta^2 \varrho_\beta(t_\beta \cdot).$$
 By Remark \ref{rem:Convergenza} and \eqref{upperbound}
%\textcolor{teal}{By Lemma  \ref{Mbeta}},
 we deduce that
	\begin{equation*}\label{complexform}
	({\rm u}_\beta, {\rm v}_\beta)
	\rightharpoonup ( {\rm u},  {\rm v}) \quad \text{in } {\rm H}_{\textrm r} \text{ as }\beta \to 0^+.
	\end{equation*}
	Since $ \| ({\rm u}_\beta, {\rm v}_\beta) \| = \| \widetilde{\varrho}_\beta \|$ and,  by Lemma \ref{unionefacili}, $\lim_{\beta\to 0^{+}} \theta_{\beta} =\pi/2$,
	we see that  ${\rm u}_\beta \to 0$ in $H^1(\mathbb{R}^3)$. Thus ${\rm u}=0$.\\	
        We claim that  $ {\rm v} \not = 0$. \\ 
	By the Sobolev embeddings and the Strauss Lemma, ${\rm u}_\beta \to  0$ %{\rm u}$ 
	and $\widetilde \varrho_\beta
	%={\rm u}_\beta/ \cos \theta_\beta
	\to  {\rm v}$ in $L^p(\mathbb{R}^3)$ for all $p\in(2,6)$. Hence, using Lemma  \ref{lowerbound}, \eqref{eq:gslevel}, and (\ref{facileiib}) in Lemma \ref{unionefacili},
		\begin{align*}
	0 < 2^{\frac{q-2}{q-1}}\mathcal I_{0,0} (\mathfrak g)
	& \leq
	I_{\lambda, \beta}({\rm u}_\beta,{\rm v}_\beta)
	=
	I_{\lambda, \beta}({\rm u}_\beta,{\rm v}_\beta) 
	- 
	\frac{1}{2}  I_{\lambda, \beta}^\prime({\rm u}_\beta,{\rm v}_\beta)
	[{\rm u}_\beta,{\rm v}_\beta]\\
	&=-\frac{\lambda}{4}
	\int\phi_{\widetilde \varrho_{\beta}} \widetilde \varrho_{\beta}^{2} 
	%(\frac{1}{|\cdot|}* (\tilde r_\beta^2))(\tilde r_\beta)
	+ \frac{q-1}{2q } \mathfrak h_\beta (y_\beta)  \|\widetilde \varrho_\beta\|^{2q}_{2q}\\
	&\leq \frac{q-1}{2q } \mathfrak h_\beta (y_\beta)  \|\widetilde \varrho_\beta\|^{2q}_{2q} 
	\longrightarrow \frac{q-1}{2q }\| {\rm v} \|^{2q}_{2q}
	\end{align*}
	getting the claim.\\
	Now we prove that $\rm v$ is a solution of \eqref{eq:ruiz}.\\	
	We know that $(\rm u_{\beta}, \rm v_{\beta})$ satisfies, for any $\varphi\in  H_{\textrm r}^{1}(\mathbb R^{3})$,
	\[
	\int \nabla \rm v_\beta \nabla \varphi + \int \rm v_\beta  \varphi
	+ \lambda \int \phi_{\rm u_{\beta}, \rm v_{\beta}}  	\rm v_\beta \varphi
	-\int |\rm v_\beta|^{2q-2} \rm v_\beta \varphi
	-\beta \int |\rm u_{\beta}|^{q} |\rm v_\beta|^{q-2} \rm v_\beta \varphi
	=0.
	\]
	Then passing to the limit as $\beta\to 0^+$, using also that, by Lemma \ref{lem:Phi},
	\[
	\left| \int  \rm v_{\beta} \varphi \phi_{\rm u_{\beta}, \rm v_{\beta} }
	- \int  \rm v \varphi \phi_{\rm u,  \rm v } \right|
	%	&\leq
	%	\left| \int (\rm u_{\beta} - \rm u) \varphi \phi_{\rm u_{\beta}, \rm v_{\beta} } \right|
	%	+ \left| \int \rm u \varphi (\phi_{\rm u_{\beta}, \rm v_{\beta} } - \phi_{\rm u,  \rm v }) \right|\\
	\leq
	\big(\|\phi_{{\rm u}_{\beta},{\rm v}_\beta }\|_{6}
	\|{\rm v}_{\beta}-{\rm v}\|_{12/5}
	+
	\|\phi_{{\rm u}_{\beta},{\rm v}_\beta } - \phi_{\rm u,  \rm v } \|_{6}
	\|{\rm v}\|_{12/5}\big)\|\varphi \|_{12/5} \to 0,
	\]
	we infer
	\begin{equation*}\label{weaku}
	\int \nabla \rm v \nabla \varphi + \int \rm v  \varphi
	+ \lambda \int \rm v\varphi \phi_{\rm u, \rm v}
	-\int|\rm v|^{2q-2} \rm v \varphi
	=0
	\end{equation*}
	which means that $\textrm v$ solves \eqref{eq:ruiz}.\\
	To show the strong convergence ${\rm v}_\beta \to {\rm v}$ in $H^{1}_{\textrm r}(\mathbb R^{3})$, observe that, for all $\psi\in H^{1}_{\textrm r}(\mathbb R^{3})$,
	\begin{equation*}\label{eq:nonloso}
	I'_{\lambda,\beta}({\rm u}_\beta,{\rm v}_\beta)[0,\psi] =0.
	\end{equation*}
	Then, choosing $\psi= {\rm v}_\beta-{\rm v}$, we get
	\begin{align*}
	\int \nabla {\rm v}_\beta \nabla ({\rm v}_\beta - {\rm v})
	+&\int {\rm v}_{\beta}({\rm v}_{\beta} - {\rm v})
	+ \lambda \int {\rm v}_{\beta}({\rm v}_{\beta} - {\rm v}) \phi_{{\rm u}_{\beta},{\rm v}_\beta }
	\\
	&=
	\int |{\rm v}_{\beta}|^{2q-2} {\rm v}_{\beta} ({\rm v}_{\beta} - {\rm v})
	+ \beta\int |{\rm u}_{\beta}|^{q}|{\rm v}_{\beta}|^{q-2} {\rm v}_{\beta} ({\rm v}_{\beta} - {\rm v}).
	\end{align*}
	Passing to the limit as $\beta\to0^+$ in the above identity,
%	taking into account also Lemma
%	\todo[inline]{Richiamare la limitatezza sulla $\phi$.}
%	, since
%	\begin{align*}
%	&\left| \int {\rm u}_{\beta}({\rm u}_{\beta} - {\rm u}) \phi_{{\rm u}_{\beta},{\rm v}_\beta }\right|
%	\leq
%	\|\phi_{{\rm u}_{\beta},{\rm v}_\beta }\|_{6}\|{\rm u}_{\beta}\|_{12/5}\|{\rm u}_{\beta}-{\rm u}\|_{12/5} \to 0,\\
%	&\left| \int |{\rm u}_{\beta}|^{2q-2} {\rm u}_{\beta} ({\rm u}_{\beta} - {\rm u}) \right|
%	\leq \|{\rm u}_{\beta}\|_{2q}^{2q-1}\|{\rm u}_{\beta} -{\rm u}\|_{2q} \to 0,\\
%	&\left| \int |{\rm v}_{\beta}|^{q}|{\rm u}_{\beta}|^{q-2} {\rm u}_{\beta} ({\rm u}_{\beta} - {\rm u}) \right|
%	\leq
%	\|{\rm v}_{\beta}\|_{2q}^{q} \|{\rm u}_{\beta}\|_{2q}^{q-1} \|{\rm u}_{\beta} - {\rm u}\|_{2q}\to 0,
%	\end{align*}
	we get 
	$\|{\rm v}_{\beta}\|^{2} \to \| {\rm v} \|^{2}$,
and so the strong convergence holds.\\
	Finally, using  \eqref{upperbound},
	we infer
	\begin{equation*}\label{eq:gsruiz}
	\mathfrak{n}
	>\mathfrak{m}_\beta
	=I_{\lambda,\beta}({\rm u}_{\beta}, {\rm v}_{\beta})
	\to \mathcal I_{\lambda,0}({\rm v}),
	\end{equation*}
	and then ${\rm v}$ is a ground state solution of \eqref{eq:ruiz} and \eqref{eq:mn} follows.
\end{proof}

Hence we can conclude.

\begin{proof}[Proof  of   \eqref{th:gsiiib} of Theorem \ref{th:gs} (asymptotic behaviour)]
		By Remark \ref{rem:Convergenza} and \eqref{upperbound} we have that $\{ (\mathfrak u_{\beta}, \mathfrak v_{\beta})\}$ is bounded and then
weakly convergent in $\rm H_{r}$ to some  $(\rm u^{*},\rm v^{*})$.\\
		First we prove that, actually, the convergence is strong.\\
		Indeed, since for every $\psi\in H^{1}_{\textrm r}(\mathbb R^{3})$, $I'_{\lambda,\beta}(\mathfrak u_{\beta}, \mathfrak v_{\beta})[\psi,0]=0$, then, choosing $\psi=\mathfrak u_{\beta} - {\rm u^{*}} $ we get, arguing as in the proof of Lemma \ref{th:tendearuiz}, that $\mathfrak u_{\beta} \to {\rm u^{*}}$ in $H_{\textrm r}^{1}(\mathbb R^{3})$. Analogously
		%\textcolor{red}{, since for every $\psi\in H^{1}_{r}(\mathbb R^{3})$, $I'_{\lambda,\beta}(\mathfrak u_{\beta}, \mathfrak v_{\beta})[0,\psi]=0$, choosing $\psi=\mathfrak v_{\beta} - {\rm v^{*}} $}
		we get $\mathfrak v_{\beta} \to {\rm v^{*}}$ in $H_{\textrm r}^{1}(\mathbb R^{3})$ and, arguing as in Lemma \ref{th:tendearuiz},  we see that $(\rm u^{*}, \rm v^{*})$ satisfies \eqref{system0}.\\
		On the other hand, by \eqref{eq:mn} and the strong convergence of $\{(\mathfrak u_{\beta}, \mathfrak v_{\beta})\}$ we arrive at
		\begin{equation}\label{eq:sopra}
		I_{\lambda,0}(\rm u^{*}, \rm v^{*})=\mathfrak n>0.
		\end{equation}
		Thus $(\rm u^{*},\rm v^{*})\in \rm H_{r}\setminus\{0\}$. \\
		Let us see now that $(\rm u^{*},\rm v^{*})$ is semitrivial.\\
%\textcolor{red}{		As done in the previous Lemma, passing to the limit in both the relations $I'_{\lambda,\beta}(\mathfrak u_{\beta}, \mathfrak v_{\beta})[\psi,0]=0$ and $I'_{\lambda,\beta}(\mathfrak u_{\beta}, \mathfrak v_{\beta})[0,\psi]=0$, with $\psi\in H^{1}_{r}(\mathbb R^{3})$ arbitrary, we obtain that $(\rm u^{*}, \rm v^{*})$ satisfies \eqref{system0}. Thus, }
Using Lemma \ref{Lembeta0} and \eqref{eq:sopra}, we get
		\[
		\mathfrak n = \mathfrak m_0 \leq I_{\lambda,0} (\rm u^{*}, \rm v^{*}) = \mathfrak n.
		\]
		Hence, $(\rm u^{*}, \rm v^{*})$ is a ground state for \eqref{system0}, and so, by 
		item \eqref{th:gsi} of Theorem \ref{th:gs}, is semitrivial.
	\end{proof}

We conclude this section with a further result about 
a particular solution of our system.\\
Let us recall that, as observed in Remark \ref{rem2l}, $\mathfrak (\mathfrak z_{\beta},\mathfrak z_{\beta})$, where $\mathfrak z_{\beta}$  is a ground state solution  of \eqref{2lambdabeta}, is a solution of \eqref{system}.
Thus, in view of (\ref{th:gsiib}) of Theorem \ref{th:gs}, for $q\in[2,3)$, such a solution is not a ground state. 
The same holds also for $q\in(3/2,2)$. More precisely we have
\begin{Th}\label{A1}
	If $\beta$ is small enough, the couple $(\mathfrak z_\beta, \mathfrak z_\beta)$ is not a ground state solution of \eqref{system}.
%	Moreover
%	$$\displaystyle\lim_{\beta\to0^{+}} I_{\lambda,\beta}(\mathfrak z_\beta, \mathfrak z_\beta) \geq I_{\lambda,0}(\mathfrak z_{0}, \mathfrak z_{0}) >\mathfrak n=\mathfrak m_{0}.$$
\end{Th}

Let us start with two preliminary lemmata concerning the monotonicity 
of the ground states levels for a single equation of  type \eqref{2lambdabeta} with respect to the parameters $\lambda$ and $\beta$.
Their proofs use standard arguments.

\begin{Lem}\label{leincr}
	Let $0<\lambda_1 < \lambda_2$ and $\mathfrak{w}_i$, $i=1,2$ be  ground state solutions
	%		 between the radial functions in $H^1(\mathbb{R}^3)$ 
	of
	\[
	-\Delta u + u + \lambda_i \phi_u u = |u|^{2q-2} u,
	\text{ in } \mathbb{R}^3,\quad i=1,2.
	\]
	Then
	\[
	\mathcal{I}_{\lambda_1,0}(\mathfrak{w}_1)< \mathcal{I}_{\lambda_2,0}(\mathfrak{w}_2).
	\]
	%As a consequence
	%$\mathfrak n=I_{\lambda,0}(0,\mathfrak{w}) <I_{\lambda,0}(\mathfrak{z},\mathfrak{z})$.
\end{Lem}

\begin{proof}
	%		Setting 
	%		\[
	%		\zeta_i(t):=t^2\mathfrak{w}_i(t\cdot),\quad t\geq0,\quad i=1,2.
	%		\]
	%		we have that
	Since
	\[
	0
	= \mathcal{J}_{\lambda_2,0}(\mathfrak{w}_2)
	= \mathcal{J}_{\lambda_1,0}(\mathfrak{w}_2) + \frac{3}{4} (\lambda_2 - \lambda_1) \int  \mathfrak{w}_2^2 \phi_{\mathfrak{w}_2}
	> \mathcal{J}_{\lambda_1,0}(\mathfrak{w}_2)
	= \mathcal{J}_{\lambda_1,0}(\zeta_{\mathfrak{w}_2}(1))
	\]
	and by \eqref{Jonevart} and (\ref{fa}) in Lemma \ref{lem:calculus},  we see that there exists $t_1 \in (0,1)$ such that
	\[
	\mathcal{J}_{\lambda_1,0}(\zeta_{\mathfrak{w}_2}(t_1))=0,
	\]
	namely, $\zeta_{\mathfrak{w}_2}(t_1)\in\mathcal{N}^{\lambda_1}$.\\ Therefore
	\[
	\mathcal{I}_{\lambda_1,0}(\mathfrak{w}_1)
	\leq \mathcal{I}_{\lambda_1,0}(\zeta_{\mathfrak{w}_2}(t_1))
	< \mathcal{I}_{\lambda_2,0}(\zeta_{\mathfrak{w}_2}(t_1))
	< \mathcal{I}_{\lambda_2,0}(\zeta_{\mathfrak{w}_2}(1))
	= \mathcal{I}_{\lambda_2,0}(\mathfrak{w}_2),
	\]
	concluding the proof.
\end{proof}

\begin{Lem}\label{monotonicitygs}
	For $0\leq\beta_1<\beta_2$, let 
	$\mathfrak z_{\beta_1}, \mathfrak z_{\beta_2}$ be respective ground states
	of \eqref{2lambdabeta}. Then
	\[
	0< \mathcal I_{2\lambda, \beta_2} (\mathfrak z_{\beta_2})
	< \mathcal I_{2\lambda, \beta_1}(\mathfrak z_{\beta_1}).
	\]
\end{Lem}
\begin{proof}
	%	\textcolor{red}{Let $0\leq\beta_1<\beta_2$. Let us consider the respective ground states of Ruiz $w_{\beta_1}$ and $w_{\beta_2}$}.
	We know that
	%	\textcolor{red}{
	%	\[
	%	\mathcal J^{2\lambda}_{\beta_1} (\mathfrak z_{1})= \frac{3}{2}  \|\nabla \mathfrak z_{1} \|_2^2
	%	+ \frac{1}{2} \| \mathfrak z_{1} \|_2^2
	%	+ \frac{3}{4} \lambda \int  \phi_{\mathfrak z_{1}} %\left(\frac{1}{|\cdot|}*w_{\beta_1}^2\right) w_{\beta_1}^2
	%	\mathfrak z_{1}^{2}
	%	-\frac{4q-3}{2q}(1+\beta_1)  \|\mathfrak z_{1}\|_{2q}^{2q}=0
	%	\]
	%	and
	%	\[
	%	\mathcal J^{2\lambda}_{\beta_2} (\mathfrak z_{\beta_{2}})= \frac{3}{2}  \|\nabla \mathfrak z_{\beta_{2}} \|_2^2
	%	+ \frac{1}{2} \| \mathfrak z_{\beta_{2}} \|_2^2
	%	+ \frac{3}{4} \lambda \int  \phi_{\mathfrak z_{\beta_{2}}} 
	%	\mathfrak z_{\beta_{2}}^{2}
	%	-\frac{4q-3}{2q}(1+\beta_2)  \|\mathfrak z_{\beta_{2}}\|_{2q}^{2q}=0.
	%	\]
	%	%\[
	%	%J_{\lambda,\beta_2} (w_{\beta_2}):= \frac{3}{2}  \|\nabla w_{\beta_2} \|_2^2
	%	%+ \frac{1}{2} \| w_{\beta_2} \|_2^2
	%	%+ \frac{3}{4} \lambda \int \left(\frac{1}{|\cdot|}*w_{\beta_2}^2\right) w_{\beta_2}^2
	%	%-\frac{4q-3}{2q}(1+\beta_2)  \|w_{\beta_2}\|_{2q}^{2q}=0.
	%	%\]
	%	Thus}
	\begin{align*}
	0&=\mathcal J_{2\lambda, \beta_1} (\mathfrak z_{\beta_1})\\
	&= \frac{3}{2}  \|\nabla \mathfrak z_{\beta_1} \|_2^2
	+ \frac{1}{2} \| \mathfrak z_{\beta_1} \|_2^2
	+ \frac{3}{2} \lambda \int  \phi_{\mathfrak z_{\beta_1}}\mathfrak z_{\beta_1}^{2}	
	-\frac{4q-3}{2q}(1+\beta_1)  \|\mathfrak z_{\beta_1}\|_{2q}^{2q}\\
	&>\frac{3}{2}  \|\nabla \mathfrak z_{\beta_1} \|_2^2
	+ \frac{1}{2} \| \mathfrak z_{\beta_1} \|_2^2
	+ \frac{3}{2} \lambda \int  \phi_{\mathfrak z_{\beta_1}}\mathfrak z_{\beta_1}^{2}	
	-\frac{4q-3}{2q}(1+\beta_2)  \|\mathfrak z_{\beta_1}\|_{2q}^{2q}\\
	%&>
	%\frac{3}{2}  \|\nabla w_{\beta_1} \|_2^2
	%+ \frac{1}{2} \| w_{\beta_1} \|_2^2
	%+ \frac{3}{4} \lambda \int \left(\frac{1}{|\cdot|}*w_{\beta_1}^2\right) w_{\beta_1}^2
	%-\frac{4q-3}{2q}(1+\beta_2)  \|w_{\beta_1}\|_{2q}^{2q}\\
	&=\mathcal J_{2\lambda, \beta_{2}} (\mathfrak z_{\beta_1}).
	\end{align*}
	Hence, by \eqref{Jonevart} and (\ref{fa}) in Lemma \ref{lem:calculus},
	%if we consider the curve
	%\todo[inline]{L'abbiamo definita ed utilizzata tantissime volte.}
	%	$
	%	\zeta_{\mathfrak z_{\beta_1}}(t):= t^2 \mathfrak z _{\beta_1}(t\cdot),
	%	$
	there exists $t_{\beta_1}\in(0,1)$ such that
	$$\mathcal J_{2\lambda,\beta_2} (\zeta_{\mathfrak z_{\beta_1}}(t_{\beta_1})) =0.$$
	Then, 
	%	\todo[inline]{Qui potremmo richiamare il Lemma sulla funzione di una variabile e anche il fatto del massimo.}
	\[
	0< \mathcal I_{2\lambda, \beta_2} (\mathfrak z_{\beta_2})
	\leq \mathcal I_{2\lambda, \beta_2}(\zeta_{\mathfrak z_{\beta_1}}(t_{\beta_1}))
	< \mathcal I_{2\lambda, \beta_1}(\zeta_{\mathfrak z_{\beta_1}}(t_{\beta_1}))
	<\mathcal I_{2\lambda, \beta_1}(\zeta_{\mathfrak z_{\beta_1}}(1))
	= \mathcal I_{2\lambda, \beta_1}(\mathfrak z_{\beta_1})
	\]
	and the proof is complete.
\end{proof}

In particular 
Lemma \ref{monotonicitygs} says that, if $\beta>0$,
\begin{equation}\label{eq:0b}
\mathcal I_{2\lambda, \beta} (\mathfrak z_{\beta})
< \mathcal I_{2\lambda, 0}(\mathfrak z_{0}).
\end{equation}

Then we can prove the desired result.

\begin{proof}[Proof of Theorem \ref{A1}]
	By Lemma \ref{leincr}, with $\lambda_{1}=\lambda , \lambda_{2} = 2\lambda, \mathfrak w_{1} = \mathfrak w, \mathfrak w_{2} = \mathfrak z_{0}$, and \eqref{i2i}, we deduce
	\begin{equation}\label{nmin}
	\mathfrak n= %I_{\lambda,0}(0,\mathfrak{w})
	\mathcal{I}_{\lambda,0}(\mathfrak{w})
	<\mathcal{I}_{2\lambda,0}(\mathfrak{z}_{0})
	<2 \mathcal{I}_{2\lambda,0}(\mathfrak{z}_{0})
	%=I_{\lambda,0}(\mathfrak{z}_{0},\mathfrak{z}_{0}),
	=I_{\lambda,0}(\mathfrak z_{0}, \mathfrak z_{0}). 
	\end{equation}
	Let now
	$\{\beta_n\}\subset(0,+\infty)$ be such that $\beta_n \to 0^{+}$ and $\beta_{n+1}<\beta_n$ and $k_{\beta_n}:=\mathcal I_{2\lambda ,\beta_n} (\mathfrak z_{\beta_n})>0$.\\
	By Lemma \ref{monotonicitygs} we know that $\{k_{\beta_{n}}\}$ is bounded and, by \eqref{eq:0b},
	\begin{equation}\label{eq:kkk}
	0<k_{\beta_0} < k_{\beta_n} < \mathcal I_{2\lambda, 0}(\mathfrak z_{0}).	
	\end{equation}
	Arguing for the single equation \eqref{2lambdabeta} 
	as in \eqref{eq:limitazione} and Remark \ref{rem:Convergenza},
	we get that $\{\mathfrak z_{\beta_n}\}$ is  bounded in 
	$H_{\textrm r}^1(\mathbb R^{3})$ and we know  also that $\mathcal I'_{2\lambda, \beta_n}(\mathfrak z_{\beta_n})=0$.
	Thus
	\begin{equation*}\label{eq:kbn}
	\mathcal I_{2\lambda, 0}(\mathfrak z_{\beta_n})
	=
	%	\mathcal I_{2\lambda, \beta_n}(\mathfrak z_{\beta_n}) 
	k_{\beta_{n}}+\frac{\beta_n}{2q} \|\mathfrak z_{\beta_n}\|_{2q}^{2q}=k_{\beta_n}+\varepsilon_n
	\end{equation*}
	%\textcolor{red}{and so $\{\mathcal I^{2\lambda}_{0}(\mathfrak z_{\beta_n})\}$ is bounded} 
	and 
	\[
	\| \mathcal I'_{2\lambda, 0}(\mathfrak z_{\beta_n})\|
	=
	%\sup_{\|v\|\leq 1} |I'_{2\lambda,0}(w_{\beta_n})[v]|
	%=
	\sup_{\|v\|\leq 1} \left| \mathcal I'_{2\lambda, \beta_n}(\mathfrak z_{\beta_n})[v] + \beta_n \int |\mathfrak z_{\beta_n}|^{2q-2}\mathfrak z_{\beta_n} v\right|
	\leq
	C \beta_n \|\mathfrak z_{\beta_n}\|^{2q-1}
	=\varepsilon_n,
	\]
	namely that $\{\mathfrak z_{\beta_n}\}$ is a (PS) sequence for $\mathcal I_{2\lambda, 0}$.\\
	Arguing as in the Proof of Proposition \ref{th:tendearuiz}, we show that $\mathfrak z_{\beta_n}\to w$ in $H_\textrm{r}^1(\mathbb R^{3})$, $\mathcal I'_{2\lambda,0}(w)=0$, and, by \eqref{eq:kkk}, $w\neq 0$.\\
	Moreover,
	%by
	%Namely $w$ is a solution of \eqref{2lambdabeta} and 
	%\todo[inline]{Non mi \`e chiaro chie \`e $\beta$.}
	%$\mathcal I_{2\lambda, 0}(\mathfrak z_{\beta})\leq \mathcal I_{2\lambda, 0}(w)$ and 
	%by \eqref{eq:kbn} we infer that
	\[
	%	\mathcal I_{2\lambda, \beta_n}(\mathfrak z_{\beta_n})
	k_{\beta_{n}}=\mathcal I_{2\lambda, 0}(\mathfrak z_{\beta_n})-\frac{\beta_n}{2q} \|\mathfrak z_{\beta_n}\|_{2q}^{2q} \to \mathcal I_{2\lambda, 0}(w).
	%\geq \mathcal I_{2\lambda, 0}(\mathfrak z_{0}).
	\]
	Hence, by \eqref{i2i} and \eqref{nmin},
	\[
	I_{\lambda, \beta_n} (\mathfrak z_{\beta_n},\mathfrak z_{\beta_n})
	=2 k_{\beta_{n}}
	\to 2 \mathcal I_{2\lambda, 0}(w)
	\geq 2 \mathcal I_{2\lambda, 0}(\mathfrak z_{0})
	=I_{\lambda, 0} (\mathfrak z_{0},\mathfrak z_{0})
	>\mathfrak{n}.
	\]
	Thus, by Proposition \ref{th:tendearuiz}, we get that for $\beta$ small  $(\mathfrak z_{\beta}, \mathfrak z_{\beta})$
	cannot be  a ground state.
%	\\Let $\varepsilon=\frac{2\mathcal I_{2\lambda, 0}(w) - \mathfrak{n}}{2}$.\\
%	There exists $\bar{\beta}>0$ such that for every $\beta <\bar{\beta}$, $\mathfrak{n}-\varepsilon<\mathfrak{m}_\beta<\frac{2\mathcal I_{2\lambda, 0}(w) + \mathfrak{n}}{2}$.\\
%	There exists $\bar{\bar{\beta}}>0$ such that for every $\beta <\bar{\bar{\beta}}$, $\frac{2\mathcal I_{2\lambda, 0}(w) + \mathfrak{n}}{2}<I_{\lambda,\beta}(\mathfrak{z}_\beta,\mathfrak{z}_\beta)<2\mathcal I_{2\lambda, 0}(w)-\varepsilon$.\\
%	Let $\tilde{\beta}:=\min\{\bar{\beta},\bar{\bar{\beta}}\}$.\\
%	If there exists $\hat{\beta}< \tilde{\beta}$ such that $I_{\lambda,\hat{\beta}}(\mathfrak{z}_{\hat{\beta}},\mathfrak{z}_{\hat{\beta}})=\mathfrak{m}_{\hat{\beta}}$, then
%	\[
%	\frac{2\mathcal I_{2\lambda, 0}(w) + \mathfrak{n}}{2}<I_{\lambda,{\hat{\beta}}}(\mathfrak{z}_{\hat{\beta}},\mathfrak{z}_{\hat{\beta}})
%	=\mathfrak{m}_{\hat{\beta}}
%	<\frac{2\mathcal I_{2\lambda, 0}(w) + \mathfrak{n}}{2}
%	\]
%	which is a contradiction.
\end{proof}

\section{The case $\beta$ large}\label{sec:blarge}

In this section we study the {\em vectorial nature} of the ground states $(\mathfrak u_{\beta}, \mathfrak v_{\beta})$ of \eqref{system} for $\beta$ large, namely satisfying \eqref{betalarge}, and
we show that such a ground state vanishes as $\beta \to +\infty$. Indeed we have

%\textcolor{teal}{More precisely we have
%\begin{Th}\label{th:blarge}
%For $\beta$ sufficiently large,  the ground states $(\mathfrak u_{\beta}, \mathfrak v_{\beta})$ of \eqref{system} are not semitrivial, {\color{blue}vectorial} and
%\begin{equation}
%\label{betainfinito}
%\lim_{\beta\to+\infty} (\mathfrak u_\beta , \mathfrak v_\beta) =0
%\quad
%\text{ in } {\rm H}.
%\end{equation}
%\end{Th}
%}
\begin{proof}[Proof of \eqref{th:gsiii} of Theorem \ref{th:gs}]

The fact that the ground state solution has to be vectorial, follows taking into account
that, by (\ref{facileiib}) and (\ref{ciiib}) of (\ref{facileiiib}) in Lemma \ref{unionefacili},
$\max_{[0,1]}\mathfrak h_{\beta}>1$. Then arguing as in Lemma \ref{Lemma52}, this implies that $\mathfrak n>\mathfrak m_{\beta}$
and so we can conclude as in the proof of \eqref{th:gsiib} of Theorem \ref{th:gs}.\\
To prove \eqref{betainfinito}, let us fix $u\in H^{1}_{\textrm r}(\mathbb R^{3})\setminus\{0\}$ and let
%\begin{eqnarray*}%\label{}
%J_{\lambda,\beta}({\gamma}_{u,u}(t)) &=&
%3 t^3\|\nabla u \|_2^2 
%+ t \| u \|_2^2
%+ 3 \lambda t^3 \int \phi_{u} u^2
%- \frac{4q-3}{q} (1+\beta) t^{4q-3}\|u\|_{2q}^{2q}.
%\end{eqnarray*}
%Let
$T_{\beta}>0$  be the real number such that $J_{\lambda,\beta} ({\gamma}_{u,u}(T_\beta))=0$, namely  such that ${\gamma}_{u,u}(T_\beta)\in \mathcal{M}$.
By (\ref{fc}) in Lemma \ref{lem:calculus} we have that
\[
\lim_{\beta\to+\infty} T_{\beta}=0
\text{ and }
\lim_{\beta\to+\infty} \frac{4q-3}{q} (1+\beta) \|u\|_{2q}^{2q}T_{\beta}^{4q-4}
=\|u\|_{2}^{2}.
\]
Thus
\begin{align*}
0
&<
\mathfrak m_\beta
\leq
I_{\lambda,\beta}({\gamma_{u,u}}(T_\beta))\\
&=
T_\beta^3\|\nabla u \|_2^2 
+ T_\beta \| u \|_2^2
+ \lambda T_\beta^3 \int \phi_u u^2
- \frac{T_\beta^{4q-3}}{q} (1+\beta) \|u\|_{2q}^{2q}\\
&=
T_\beta \| u \|_2^2
+\frac{1}{3}
\left[
\frac{4q-3}{q} (1+\beta) T_\beta^{4q-3}\|u\|_{2q}^{2q}
- T_\beta \| u \|_2^2
\right]
- \frac{T_\beta^{4q-3}}{q} (1+\beta) \|u\|_{2q}^{2q}\\
&=
\frac{2}{3} T_\beta
\left(
\| u \|_2^2
- \frac{2q-3}{q} (1+\beta)T_\beta^{4q-4} \|u\|_{2q}^{2q}
\right)
\to 0
\end{align*}
as $\beta\to+\infty$.\\
Hence we conclude by Remark \ref{rem:Convergenza}.
\end{proof}

Now we show that, actually, a ground state can be taken with the two components equal. Indeed
\begin{Th}\label{A2}
If $\beta$ satisfies \eqref{betalarge}, then
\[
\mathfrak{m}_\beta=I_{\lambda,\beta}(\mathfrak{z}_\beta,\mathfrak{z}_\beta),
\]
where $\mathfrak{z}_\beta$ is a ground state solution  of \eqref{2lambdabeta}.
\end{Th}

Let $(\mathfrak u_{\beta}, \mathfrak v_{\beta})$ be a vectorial ground state just found and let us consider its polar coordinates as in \eqref{eq:polar}.
Taking into account Lemma \ref{Mbeta} and  so Remark \ref{rem:generale},
using (\ref{facileiib}) and (\ref{ciiib}) of (\ref{facileiiib}) in Lemma \ref{unionefacili},
we have
\begin{Lem}\label{Mbetag}
	If $\beta$ satisfies \eqref{betalarge}, then there exists $t_\beta > 0$ such that 
	$\gamma_{\varrho_{\beta}/\sqrt{2},\varrho_{\beta}/\sqrt{2}}(t_\beta)
	\in \mathcal{M}$
	and
	\begin{equation*}\label{eq:gslevelg}
	\mathfrak m_{\beta}
	= I_{\lambda,  \beta}(\gamma_{\varrho_{\beta}/\sqrt{2},\varrho_{\beta}/\sqrt{2}}(t_\beta)).
	\end{equation*}
	In particular $\gamma_{\varrho_{\beta}/\sqrt{2},\varrho_{\beta}/\sqrt{2}}(t_\beta)$ is a  ground state solution.
\end{Lem}
Thus we are ready to complete the proof.

\begin{proof}[Proof of Theorem \ref{A2}]
By Lemma \ref{Mbetag}, there exists $\mathfrak{u}_\beta \in H_{\rm r}^1(\mathbb{R}^3)\setminus\{0\}$
such that
\[
\mathfrak{m}_\beta=I_{\lambda,\beta}(\mathfrak{u}_\beta,\mathfrak{u}_\beta).
\]
Thus, by Remark \ref{rem2l}, we infer
\[
\mathfrak{m}_\beta=2\mathcal{I}_{2\lambda,\beta}(\mathfrak{u}_\beta)
\geq 2\mathcal{I}_{2\lambda,\beta}(\mathfrak{z}_\beta)
= I_{\lambda,\beta}(\mathfrak{z}_\beta,\mathfrak{z}_\beta)
\geq \mathfrak{m}_\beta
\]
concluding the proof.
\end{proof}

\section{The particular case $\beta=2^{q-1}-1$ and $q\in[2,3)$}\label{sec:7}

In this particular case we can argue as in Section \ref{sec:POSITIVESMALL}
and Section \ref{sec:blarge} to get both (semitrivial and vectorial) types of ground states.
\\
By Lemma \ref{Lemma51}, being $\mathfrak m_{\beta}=\mathfrak n$,
we get that 
$$I_{\lambda,\beta}(\mathfrak w,0)=I_{\lambda,\beta}(0,\mathfrak w)=\mathfrak m_{\beta}$$
and so $(\mathfrak w,0)$ and $(0,\mathfrak w)$ are semitrivial ground states.\\
Of course in this case we cannot proceed as in the proof of 
item \eqref{th:gsi} of Theorem \ref{th:gs} since $0$ and $1$ are not the only maximisers
of $\mathfrak h_{\beta}$ and so, the existence of further maximisers
gives   vectorial ground states too
(see (\ref{biiib}) of item (\ref{facileiiib}) in Lemma \ref{unionefacili}).
\\
Indeed Lemma \ref{Mbeta} applies with $y_{\beta}=1/2$,
and so $\theta_{\beta}=\pi/4$. Hence we get that 
there exists $t_\beta > 0$ such that 
	$\gamma_{\varrho_{\beta}/\sqrt{2},\varrho_{\beta}/\sqrt{2}}(t_\beta)
	\in \mathcal{M}$,
	\begin{equation*}%\label{eq:gslevelg}
	\mathfrak m_{\beta}
	= I_{\lambda,  \beta}(\gamma_{\varrho_{\beta}/\sqrt{2},\varrho_{\beta}/\sqrt{2}}(t_\beta)),
	\end{equation*}
and  $\gamma_{\varrho_{\beta}/\sqrt{2},\varrho_{\beta}/\sqrt{2}}(t_\beta)$ is a  ground state solution.
\\
More in particular, if $q=2$, $\mathfrak h_{\beta}\equiv1$. Then we can take an arbitrary
$y_{\beta}\in(0,1)$, obtaining that there exists $t_{\beta}>0$ such that
$\gamma_{\varrho_{\beta}\cos\theta_{\beta},\varrho_{\beta}\sin\theta_{\beta}}(t_\beta)
$ is a  ground state solution.
\\
Finally we observe that, as a corollary of this last property, arguing as in Theorem \ref{A2}, we have also that
\[
\mathfrak{m}_\beta=I_{\lambda,\beta}(\mathfrak{z}_\beta,\mathfrak{z}_\beta).
\]

\appendix

\section{Proof of Lemma \ref{unionefacili}}\label{AppLemma24}

In this Appendix we present the details of the proof of Lemma \ref{unionefacili}.\\
	First observe that $\mathfrak{h}_\beta$ is even with respect to the line $y=1/2$.\\
	%{\color{red}, thus, to simplify the proofs, we can restrict ourselves to the interval $[0,1/2]$}.\\
	Since $\mathfrak{h}_0$ is strictly decreasing in $[0,1/2]$, property (\ref{facilei}) is trivial. %{\color{brown}In particular we observe that its global minimum is $\mathfrak{h}_0({1}/{2})=2^{1-q}$.}
	\\
	%	Moreover have
	%	$$\mathfrak{h}_0'(y) = qy^{q-1}-q(1-y)^{q-1} \text{ in } [0,1]$$
	%	and
	%	$$\mathfrak{h}_0''(y) = q(q-1)y^{q-2} + q(q-1)(1-y)^{q-2} \text{ in } (0,1),$$
	%	so that
	%	$$
	%	\mathfrak{h}_0'(y)=0 \Leftrightarrow y=1/2 \quad \text{ and } \quad  \mathfrak{h}_0''({1}/{2}) >0,
	%	$$
	%	concluding the proof.\\
	%{\color{red}If $q\geq 2$, being $\mathfrak{h}_\beta$ the sum of $\mathfrak{h}_0$ plus $\beta$ times a positive bounded function, to get (\ref{facileii}) it is enough to observe that $\mathfrak{h}_\beta(0)=1$ and $\mathfrak{h}_\beta'(0)<0$ (for $\beta$ small if $q=2$).\\}
	Now let us consider $\beta>0$. The proof when $q=2$ is trivial. Thus, let us focus on (\ref{facileiiib}) for $q\in (2,3)$ 
	and (\ref{facileiib}).\\
	Observe that, for any fixed $\beta>0$, we have that $\mathfrak{h}_\beta'(1/2)=0$ and $\mathfrak{h}_\beta''(1/2)=2^{3-q}q(q-1-\beta)$. Thus, if $\beta<q-1$, then, $1/2$ cannot be a maximum point of $\mathfrak{h}_\beta$.\\
	Moreover, in $(0,1/2]$,
	\begin{equation}\label{eq:rhs}
		\frac{\mathfrak{h}_{\beta}'(y)}{q(1-y)^{q-1}}
		=
		\Big( \frac{y}{1-y}\Big)^{q-1}
		-1
		+\beta \Big( \frac{y}{1-y}\Big)^{\frac{q}{2}-1} 
		-\beta\left( \frac{y}{1-y} \right)^\frac{q}{2}
	\end{equation}
	Thus, to study the sign of $\mathfrak{h}_{\beta}'$, we can consider the right hand side of \eqref{eq:rhs} and, for simplicity, we write it as
	\[
	\mathfrak{g}_{\beta} (t):=t^{q-1} -1 +\beta t^{q/2-1}-\beta t^{q/2},
	\quad
	t\in(0,1],
	\]
	whose derivative is
	\[
		\qquad
		\mathfrak{g}_{\beta}' (t)
		=\frac{\mathfrak{r}_\beta(t)}{2t^{2-q/2}}
		\quad
		\text{ where }
		\mathfrak{r}_\beta(t):=2(q-1)t^{q/2} - \beta qt + \beta(q-2).
		\]
	Let us prove (\ref{facileiib}).\\
	Note that, whenever $q\in(3/2,2)$ and $\beta>0$, we have that $\displaystyle \lim_{y\to0^{+}} \mathfrak{h}_{\beta}'(y)=+\infty$.
	Thus, since $\mathfrak{h}_\beta(0)=1$, we get the existence of $y_\beta\in(0,1/2]$ such that $\mathfrak{h}_\beta(y_\beta)=\max_{y\in[0,1]}\mathfrak{h}_\beta(y)>1$.\\
	Moreover
	\begin{equation}
		\label{propgii}
		\lim_{t\to 0^+} \mathfrak{g}_{\beta} (t)=+\infty
		\text{ and }
		\mathfrak{g}_{\beta} (1)=0.
	\end{equation}
	If $\beta\in(0,q-1)$ we have that $\mathfrak{r}_\beta(0)<0$, $\mathfrak{r}_\beta(1)>0$, and $\mathfrak{r}_\beta$ is (strictly) increasing on the left of its unique maximum point $((q-1)/\beta)^{2/(2-q)}$ and then it is (strictly) decreasing. Thus $\mathfrak{r}_\beta$ has a unique zero $\mathfrak{t}_\beta$ which is the unique critical point (minimizer) of $\mathfrak{g}_\beta$ and $\mathfrak{g}_\beta$ is (strictly) decreasing in $(0,\mathfrak{t}_\beta)$ and (strictly) increasing in $(\mathfrak{t}_\beta,1)$. Hence, by \eqref{propgii}, $\mathfrak{g}_\beta$ has a unique zero in $(0,1)$ which gives us the unique maximum point $y_\beta$.\\
	If $\beta=q-1$ we have that $\mathfrak{r}_\beta(0)<0$, $\mathfrak{r}_\beta(1)= 0$, and  $\mathfrak{r}_\beta$ is (strictly) increasing in $(0,1)$. Thus $\mathfrak{g}'_{\beta}$ is (strictly) negative in $(0,1)$ and so, by \eqref{eq:rhs} and \eqref{propgii}, $\mathfrak{h}_{\beta}'$ is (strictly) positive in $(0,1/2)$. Hence the symmetry of $\mathfrak{h}_{\beta}$ allows us to conclude.\\
	If $\beta> q-1$ we have that $\mathfrak{r}_\beta(0)<0$, $\mathfrak{r}_\beta(1)< 0$, and $\mathfrak{r}_\beta$ is (strictly) increasing on the left of  its unique maximum point $((q-1)/\beta)^{2/(2-q)}$ and then it is (strictly) decreasing. Moreover
	\[
	\mathfrak{r}_\beta\Big((q-1)/\beta)^{2/(2-q)}\Big)
	=
	\frac{2-q}{\beta^{q/(2-q)}}\big((q-1)^{2/(2-q)}-\beta^{2/(2-q)}\big)< 0.
	\]
	Then, $\mathfrak{r}_{\beta}$ is (strictly) negative in $(0,1)$ and so, by \eqref{eq:rhs} and \eqref{propgii}, $\mathfrak{h}_{\beta}'$ is (strictly) positive in $(0,1/2)$. Hence we can conclude as in the previuos step.\\
	To prove the asymptotic behavior of $y_\beta$ as $\beta\to 0^+$, let us recall that $y_\beta\in(0,1/2)$ for $\beta<q-1$.
	If, by contradiction, we assume that $y_\beta \not\to 0$ as $\beta \to 0^+$, then there exists a  sequence $\{\beta_{n}\}$ tending to zero, such that $\mathfrak h_{\beta_{n}} (y_{\beta_{n}})>1$ and $\displaystyle \lim_{n}y_{\beta_{n}}= \ell \in (0,1/2]$. Then
	$$
	1\leq
	\lim_{n} \mathfrak h_{\beta_{n}}(y_{\beta_{n}})=\ell^{q}+(1-\ell)^{q}
	\leq\frac{1}{2^{q-1}}
	<1 $$
	getting the contradiction.\\
	Let us prove (\ref{facileiiib}).\\
	In this case, namely whenever $q\in(2,3)$ and $\beta>0$, we have that $\mathfrak{h}_{\beta}'(0)=-q$.
	Moreover
	\begin{equation}
		\label{propgiii}
		\mathfrak{g}_{\beta} (0)=-1
		\text{ and }
		\mathfrak{g}_{\beta} (1)=0.
	\end{equation}
	If $\beta\in (0,q-1)$, we have that $\mathfrak{r}_\beta(0)>0$, $\mathfrak{r}_\beta(1)>0$, and $\mathfrak{r}_\beta$ is (strictly) decreasing before its unique minimum point $(\beta/(q-1))^{2/(q-2)}$ and then it is (strictly) increasing. Moreover 
	\[
	\mathfrak{r}_\beta\Big((\beta/(q-1))^{2/(q-2)}\Big)
	=
	\frac{\beta(q-2)}{(q-1)^{2/(q-2)}}\big((q-1)^{2/(q-2)}-\beta^{2/(q-2)}\big)> 0.
	\]
	Thus, $\mathfrak{r}_{\beta}$ is (strictly) positive in $(0,1)$ and so, by \eqref{eq:rhs} and \eqref{propgiii}, $\mathfrak{h}_{\beta}'$ is (strictly) negative in $(0,1/2)$. Hence the symmetry of $\mathfrak{h}_{\beta}$ allows us to conclude.\\
	If $\beta=q-1$ we have that $\mathfrak{r}_\beta(0)>0$, $\mathfrak{r}_\beta(1)= 0$, and $\mathfrak{r}_\beta$ is (strictly) decreasing in $(0,1)$. Thus $\mathfrak{g}'_{\beta}$ is (strictly) positive in $(0,1)$ and so, by \eqref{eq:rhs} and \eqref{propgiii}, $\mathfrak{h}_{\beta}'$ is (strictly) negative in $(0,1/2)$. Hence we can conclude as in the previuos step.\\
	If $\beta > q-1$, we have that $\mathfrak{r}_\beta(0)>0$, $\mathfrak{r}_\beta(1)< 0$, , and $\mathfrak{r}_\beta$ is (strictly) decreasing in $(0,1)$. Thus $\mathfrak{r}_\beta$ has a unique zero $\mathfrak{t}_\beta$ which is the unique critical point (maximum point) of $\mathfrak{g}_\beta$ and $\mathfrak{g}_\beta$ is (strictly) increasing in $(0,\mathfrak{t}_\beta)$ and (strictly) decreasing in $(\mathfrak{t}_\beta,1)$. Hence, by \eqref{propgiii}, $\mathfrak{g}_\beta$ has a unique zero in $(0,1)$ which gives us a unique minimum point of $\mathfrak{h}_\beta$ in $(0,1/2)$ and so, the unique local maximum point of $\mathfrak{h}_\beta$ in $(0,1)$ is $1/2$.\\
	Since $\mathfrak{h}_\beta(0)=\mathfrak{h}_\beta(1)=1$ and $\mathfrak{h}_\beta(1/2)=(1+\beta)/2^{q-1}$ we get that $1/2$ is the global maximum point of $\mathfrak{h}_\beta$ in $[0,1/2]$ if and only if $\beta\geq 2^{q-1}-1$ and it is the unique global maximum if and only if $\beta > 2^{q-1}-1$.

\subsection*{Acknowledgements}
P. d'Avenia was partially supported by PRIN 2017JPCAPN {\em Qualitative and quantitative aspects of nonlinear PDEs} and by FRA2019 of Politecnico di Bari.
L.A. Maia was partially supported by FAPDF, CAPES, and CNPq grant 308378/2017-2.
G. Siciliano  was partially supported by
Fapesp grant 2018/17264-4, CNPq grant 304660/2018-3, FAPDF, and CAPES (Brazil) and INdAM (Italy).
\\
This work was partially carried out during a stay of P. d'Avenia at Universidade de S\~{a}o Paulo, of P. d'Avenia and G. Siciliano at Universidade de Bras\'ilia, of P. d'Avenia and L.A. Maia at Sapienza Universit\`a di Roma.
The authors would like to express their deep gratitude to the respective Universities for the warm hospitality.\\
Finally the authors thank Roberto Celiberto for the useful discussions about the Hartree-Fock method and the anonymous referees for the useful suggestions.

\appendix

\end{document}